\documentclass[12pt,reqno]{siugrad51sin}
\usepackage{amsmath, amssymb, amsthm, latexsym, mathtools,bold-extra}
\usepackage{adjustbox}
\makeatletter
\providecommand*{\dashv}{%
  \mathrel{%
    \mathpalette\@dashv\vdash
  }%
}
\newcommand*{\@dashv}[2]{%
  \reflectbox{$\m@th#1#2$}%
}
\usepackage{tikz-cd}
\usepackage[colorlinks=false,hyperfootnotes=false]{hyperref}
\renewcommand*\thesection{\arabic{section}}%%%%
\theoremstyle{plain}
\newtheorem{theorem}{Theorem}[section]
\newtheorem{mainthm}[theorem]{Main Theorem}
\newtheorem{corollary}[theorem]{Corollary}
\newtheorem{lemma}[theorem]{Lemma}
\newtheorem{proposition}[theorem]{Proposition}
\theoremstyle{definition}
\newtheorem{example}[theorem]{Example}
\newtheorem{fact}[theorem]{Fact}
\newtheorem{definition}[theorem]{Definition}
\newtheorem{remark}[theorem]{Remark}
\newtheorem{assum}[theorem]{Assumption}
\newtheorem{ques}[theorem]{Question}
\newtheorem{obs}[theorem]{Observation}

\newtheorem*{theorem*}{Theorem}

\newcommand{\lk}{\leq_{\bf K}}
\newcommand{\lkk}{\leq_{\bf K_1}}
\newcommand{\mn}{\mathfrak{C}}
\newcommand{\oop}{\operatorname}
\newcommand{\gtp}{\oop{gtp}}
\newcommand{\gs}{\oop{gS}}
\newcommand{\cf}{\oop{cf}}

\newcommand{\ls}{\oop{LS}({\bf K})}

\newcommand{\llk}{\oop{L}({\bf K})}
\newcommand{\defeq}{\vcentcolon=}

\def\fork{\mathrel{\raise0.2ex\hbox{\ooalign{\hidewidth$\vert$\hidewidth\cr\raise-0.9ex\hbox{$\smile$}}}}}
\newcommand{\fk}[3]{#1 \underset{#2}{\fork} #3}
\newcommand{\fkc}[3]{#1 \underset{#2}{\bar{\fork}} #3}

\newcommand{\nr}[1]{\lVert #1 \rVert}

\newcommand{\al}{{\aleph_0}}

\DeclareMathOperator{\tp}{tp}

\makeatletter
\@ifpackageloaded{hyperref}%
  {\newcommand{\mylabel}[2]% #1=name, #2 = contents
    {\protected@write\@auxout{}{\string\newlabel{#1}{{#2}{\thepage}%
      {\@currentlabelname}{\@currentHref}{}}}}}%
  {\newcommand{\mylabel}[2]% #1=name, #2 = contents
    {\protected@write\@auxout{}{\string\newlabel{#1}{{#2}{\thepage}}}}}
\makeatother
\begin{document}

\pagenumbering{roman}
\setcounter{page}{0}
\newpage
\pagenumbering{arabic}
\setcounter{page}{1}
%\raggedtright
\parindent=.35in
\begin{center}
        \begin{center}%
         {\Large\bfseries\scshape Stability results assuming tameness,\\monster model and continuity of nonsplitting}\\\vspace{1em}{\scshape Samson Leung}\\      
        \end{center}%
\end{center}
{\let\thefootnote\relax\footnote{Date: \today
}} 
\begin{abstract}
%In this chapter
Assuming the existence of a monster model, tameness and continuity of nonsplitting in an abstract elementary class (AEC), we extend known superstability results: let $\mu>\ls$ be a regular stability cardinal and let $\chi$ be the local character of $\mu$-nonsplitting. The following holds:
\begin{enumerate}
\item When $\mu$-nonforking is restricted to $(\mu,\geq\chi)$-limit models ordered by universal extensions, it enjoys invariance, monotonicity, uniqueness, existence, extension and continuity. It also has local character $\chi$. This generalizes Vasey's result \cite[Corollary 13.16]{snote} which assumed $\mu$-superstability to obtain same properties but with local character $\aleph_0$.
\item There is $\lambda\in[\mu,h(\mu))$ such that if ${\bf K}$ is stable in every cardinal between $\mu$ and $\lambda$, then ${\bf K}$ has $\mu$-symmetry while $\mu$-nonforking in (1) has symmetry. In this case
\begin{enumerate}
\item ${\bf K}$ has the uniqueness of $(\mu,\geq\chi)$-limit models: if $M_1,M_2$ are both $(\mu,\geq\chi)$-limit over some $M_0\in K_\mu$, then $M_1\cong_{M_0}M_2$;
\item any increasing chain of $\mu^+$-saturated models of length $\geq\chi$ has a $\mu^+$-saturated union. These generalize \cite{vv} and remove the symmetry assumption in \cite{bvulm,s19} .
\end{enumerate}\end{enumerate}
Under $(<\mu)$-tameness, the conclusions of (1), (2)(a)(b) are equivalent to ${\bf K}$ having the $\chi$-local character of $\mu$-nonsplitting. 

Grossberg and Vasey \cite{GV,s19} gave eventual superstability criteria for tame AECs with a monster model. We remove the high cardinal threshold and reduce the cardinal jump between equivalent superstability criteria. We also add two new superstability criteria to the list: a weaker version of solvability and the boundedness of the $U$-rank.
\end{abstract}
\vspace{1em}
\begin{center}
{\bfseries TABLE OF CONTENTS}
\end{center}
\tableofcontents%chapter
\vspace{2em}
 
\section{Introduction}\label{ufsec1}
Good frames in abstract elementary classes (AECs) were constructed in \cite[IV Theorem 4.10]{shh}, assuming categoricity and non-ZFC axioms. Later Boney and Grossberg \cite{BG} built a good frame from coheir with the assumption of tameness and extension property of coheir in ZFC. Vasey \cite[Section 5]{s5} further developed on coheir and \cite{s6} managed to construct a good frame at a high categoricity cardinal (categoricity can be replaced by superstability and type locality, but the initial cardinal of the good frame is still high). 

Another approach to building a good frame is via nonsplitting. It is in general not clear whether uniqueness or transitivity hold for nonsplitting (where models are ordered by universal extensions). To resolve this problem, Vasey \cite{s3} constructed nonforking from nonsplitting, which has nicer properties: assuming superstability in $K_\mu$, tameness and a monster model, nonforking gives rise to a good frame over the limit models in $K_{\mu^+}$ \cite[Corollary 6.14]{vv}. Later it was found that uniqueness of nonforking also holds for limit models in $K_\mu$ \cite{s22}. 

We will generalize the nonforking results by replacing the superstability assumption by continuity of nonsplitting. A key observation is that the extension property of nonforking still holds if we have continuity of nonsplitting and stability. This allows us to replicate extension, uniqueness and transitivity properties. Since the assumption of continuity of nonsplitting applies to universal extensions only, we only get continuity and local character for universal extensions. Hence we can build an approximation of a good frame which is over the skeleton (see \ref{skedef}) of long enough limit models ordered by universal extensions. We state the known result and our result for comparison:

\begin{theorem}
Let $\mu\geq\ls$, ${\bf K}$ have a monster model, be $\mu$-tame and stable in $\mu$. Let $\chi$ be the local character of $\mu$-nonsplitting.  
\begin{enumerate}
\item \cite[Corollary 13.16]{snote} If ${\bf K}$ is $\mu$-superstable, then there exists a good frame over the skeleton of limit models in $K_\mu$ ordered by $\leq_u$, except for symmetry;
\item (\ref{gfcor}) If $\mu$ is regular and ${\bf K}$ has continuity of $\mu$-nonsplitting, then there exists a good $\mu$-frame over the skeleton of $(\mu,\geq\chi)$-limit models ordered by $\leq_u$, except for symmetry. The local character is $\chi$ in place of $\al$. 
\end{enumerate}
\end{theorem}
We assumed that $\mu$ is regular to guarantee that $\chi\leq\mu$. In the superstable case, $\chi=\al\leq\mu$ by the definition of $\mu$-superstability.

To obtain symmetry for our frame, we look at the argument in \cite{vv}. In \cite{van16a,van16b}, VanDieren defined a stronger version of symmetry called $\mu$-symmetry and proved its equivalence with the continuity of reduced towers. \cite[Lemma 4.6]{vv} noticed that a weaker version of symmetry is sufficient in one direction and deduced the weaker version of symmetry via superstability. To generalize these arguments, in Section \ref{ufsec5} we replace superstability by continuity of nonsplitting and stability in a range of cardinals (the range depends on the no-order-property of ${\bf K}$, see \ref{sympf}). Then we can obtain a local version of $\mu$-symmetry, which implies symmetry of our frame for long enough limit models. Notice that in the superstable case, $\chi=\al$ while $(\mu,\chi)$-symmetry is the same as $\mu$-symmetry.
\begin{theorem}
Let $\mu\geq\ls$, ${\bf K}$ be $\mu$-tame and stable in $\mu$. Let $\chi$ be the local character of $\mu$-nonsplitting.  
\begin{enumerate}
\item \cite[Corollary 6.9]{vv} If ${\bf K}$ is $\mu$-superstable, then it has $\mu$-symmetry;
\item (\ref{symcor2}) If $\mu$ is regular and ${\bf K}$ has continuity of $\mu$-nonsplitting. There is $\lambda<h(\mu)$ such that if ${\bf K}$ is stable in every cardinal between $\mu$ and $\lambda$, then ${\bf K}$ has $(\mu,\chi)$-symmetry.
\end{enumerate}
\end{theorem}

Continuity of nonsplitting and the localization of symmetry were already exploited in \cite[Theorem 20]{bvulm} to obtain the uniqueness of long enough limit models (see \ref{ulmfact}). They simply assumed the local symmetry while we used the argument in \cite{vv} to deduce it from extra stability and continuity of nonsplitting (\ref{unicor}). On the other hand, \cite[Section 11]{s19} used continuity of nonsplitting to deduce that a long enough chain of saturated models of the same cardinality is saturated. There he assumed saturation of limit models and managed to satisfy this assumption using his earlier result with Boney \cite{bv2}, which has a high cardinal threshold. Since we already have local symmetry under continuity of nonsplitting and extra stability, we immediately have uniqueness of long limit models, and hence Vasey's argument can be applied to obtain the above result of saturated models (see \ref{s19prop}; a comparison table of the approaches can be found in \ref{comparermk}(2)).

Vasey \cite[Lemma 11.6]{s19} observed that a localization of VanDieren's result \cite{van16a} can give: if the union of a long enough chain of $\mu^+$-saturated models is $\mu^+$-saturated, then local symmetry is satisfied. Assuming more tameness, we use this observation to obtain converses of our results (see \ref{maint1}(4)$\Rightarrow$(3)). In particular local symmetry will lead to uniqueness of long limit models, which implies local character of nonsplitting (\ref{maint1}(3)$\Rightarrow$(1)). Despite the important observation by Vasey, he did not derive these corollaries.

\begin{theorem}
Let $\mu>\ls$, $\delta\leq\mu$ be regular, ${\bf K}$ have a monster model, be $(<\mu)$-tame, stable in $\mu$ and has continuity of $\mu$-nonsplitting. If any increasing chain of $\mu^+$-saturated models of cofinality $\geq\delta$ has a $\mu^+$-saturated union, then ${\bf K}$ has $\delta$-local character of $\mu$-nonsplitting.
\end{theorem}

The equivalent properties of a stable AEC with continuity of nonsplitting can be specialized to a superstable AEC, because superstability implies stability and continuity of nonsplitting. In \cite{GV}, equivalent superstability properties were listed using the machinery of averages, leading to a high cardinal threshold for the equivalences to take place, and a high cardinal jump when moving from one property to another. In comparison, the equivalent properties we obtained in \ref{maint1} and \ref{maint2} do not require a high cardinal threshold (simply $\mu>\ls$ to make sense of saturated models) but we do need extra stability assumptions above $\mu$. Such stability assumption can be replaced by a smaller range of stability plus more no-order-property. Except for transferring stability in a cardinal to superstability, all other properties are equivalent to each other up to a jump to the successor cardinal.

In the original list inside \cite{GV}, $(\lambda,\xi)$-solvability was considered for $\lambda>\xi$, which they showed to be an equivalent definition of superstability, with a huge jump of cardinal from no long splitting chains to solvability. Further developments in \cite{s17} indicate that such solvability has downward transfer properties which seem too strong to be called superstability. We propose a variation where $\lambda=\xi$ and will prove its equivalence with no long splitting chains in the same cardinal above $\mu^+$ (under continuity of nonsplitting and stability). At $K_\mu$, we demand $(<\mu)$-tameness for the equivalence to hold, up to a jump to the successor cardinal.
\begin{theorem}
Let $\mu>\ls$, ${\bf K}$ have a monster model, be $(<\mu)$-tame, stable in $\mu$.
\begin{enumerate}
\item \cite{sv99} If there is $\lambda>\mu$ such that ${\bf K}$ is $(\lambda,\mu)$-solvable, then it is $\mu$-superstable;
\item \cite[Corollary 5.5]{GV} If $\mu$ is high enough and ${\bf K}$ is $\mu$-superstable, then there is some $\lambda\geq\mu$ and some $\lambda'<\lambda$ such that ${\bf K}$ is $(\lambda,\lambda')$-solvable;
\item (\ref{supsolv}) If ${\bf K}$ has continuity of $\mu$-nonsplitting, then it is $\mu$-superstable iff it is $(\mu^+,\mu^+)$-solvable.
\end{enumerate}
\end{theorem}

Meanwhile, \cite[Corollary 4.24]{s19} showed that stability in a tail is also an equivalent definition of superstability, but the starting cardinal of superstability $(\lambda'({\bf K}))^++\chi_1$ is only bounded by the Hanf number of $\mu$. Since we assume continuity of nonsplitting, we can obtain $\mu^{ }$-superstability by assuming stability in unboundedly many cardinals below $\mu^{ }$, and enough stability above $\mu^{ }$. 

\begin{theorem}
Let $\mu>\ls$ with cofinality $\al$, ${\bf K}$ have a monster model, be $\mu$-tame, stable in both $\mu$ and unboundedly many cardinals below $\mu$.
\begin{enumerate}
\item \cite[Corollary 4.14]{s19} If $\mu\geq(\lambda'({\bf K}))^++\chi_1$, then ${\bf K}$ is $\mu$-superstable;
\item (\ref{cfcal}) If ${\bf K}$ has continuity of $\mu$-nonsplitting, then there is $\lambda<h(\mu)$ such that if ${\bf K}$ is stable in $[\mu,\lambda)$, then it is $\mu$-superstable.
\end{enumerate}
\end{theorem}

As the final item of the list, we prove that the boundedness of the $U$-rank (with respect to $\mu$-nonforking for limit models in $K_\mu$ ordered by universal extensions) is equivalent to $\mu$-superstability (\ref{urankcor}). We will need to extend our nonforking to longer types, using results from \cite{bv}. Then we can quote a lot of known results from \cite{BG}, \cite{BGKV} and \cite{GM}. Our strategy of extending frames contrasts with \cite{s6} which used a complicated axiomatic framework and drew technical results from \cite[III]{shh}. Here we directly construct a type-full good $\mu$-frame from nonforking and the known results apply (which are independent of the technical ones in \cite{s6,shh}).

\begin{theorem}
Let $\mu\geq\ls$ be regular, ${\bf K}$ have a monster model, be $\mu$-tame, stable in $\mu$ and have continuity of $\mu$-nonsplitting. Let $U(\cdot)$ be the $U$-rank induced by $\mu$-nonforking restricted to limit models in $K_\mu$ ordered by $\leq_u$. The following are equivalent:
\begin{enumerate}
\item ${\bf K}$ is $\mu$-superstable;
\item $U(p)<\infty$ for all $p\in\gs(M)$ and limit model $M\in K_\mu$.
\end{enumerate}
\end{theorem}

In Section \ref{ufsec2}, we will state our global assumptions; define limit models, skeletons and good frames. In Section \ref{ufsec3}, we will review useful properties of nonsplitting with miscellaneous improvements. In Section \ref{ufsec4}, we will use $\mu$-nonforking to construct our good frame over the skeleton of $(\mu,\geq\chi)$-limit models ordered by $\leq_u$, except for two changes: the local character of the frame will be $\chi$ in place of $\al$, while symmetry properties will be proven in Section \ref{ufsec5} under extra stability assumptions. In Section \ref{ufsec6}, we will generalize known superstability results using the symmetry properties. In particular we guarantee that the union of $\mu^+$-saturated models is saturated, provided that we have extra stability, continuity of nonsplitting and the chain being long enough. In Section \ref{ufsec7}, we will consider two characterizations of superstability, stability in a tail and the boundedness of the $U$-rank. We will prove the main theorems in Section \ref{ufsec8} and state two applications there.

This paper was written while the author was working on a Ph.D. under the direction of Rami Grossberg at Carnegie Mellon University and we would like to thank Prof. Grossberg for his guidance and assistance in my research in general and in this work in particular. We also thank John Baldwin and Marcos Mazari-Armida for useful comments. 

\section{Preliminaries}\label{ufsec2}
Throughout this paper, we assume the following:
\begin{assum}
\begin{enumerate}
\item ${\bf K}$ is an AEC with $AP$, $JEP$ and $NMM$.\mylabel{assum1}{Assumption \thetheorem}
\item ${\bf K}$ is stable in some $\mu\geq\ls$.
\item ${\bf K}$ is $\mu$-tame.
\item ${\bf K}$ satisfies continuity of $\mu$-nonsplitting.
\item $\chi\leq\mu$ where $\chi$ is the minimum local character cardinal of $\mu$-nonsplitting (see \ref{chi}).
\end{enumerate}
\end{assum}
$AP$ stands for amalgamation property, $JEP$ for joint embedding property and $NMM$ for no maximal model. They allow the construction of a monster model. Given a model $M\in K$, we write $\gs(M)$ the set of Galois types over $M$ (the ambient model does not matter because of $AP$). 
\begin{definition} Let $\lambda$ be an infinite cardinal.
\begin{enumerate}
\item $\alpha\geq2$ be an ordinal, ${\bf K}$ is \emph{$(<\alpha)$-stable in $\lambda$} if for any $\nr{M}=\lambda$, $|\gs^{<\alpha}(M)|\leq\lambda$. We omit $\alpha$ if $\alpha=2$. 
\item ${\bf K}$ is \emph{$\lambda$-tame} if for any $N\in K$, any $p\neq q\in\gs(N)$, there is $M\leq N$ of size $\lambda$ such that $p\restriction M\neq q\restriction M$. 
\end{enumerate}
\end{definition}
We will define continuity of $\mu$-nonsplitting in \ref{condef}.
\begin{definition}
Let $\lambda\geq\ls$ be a cardinal and $\alpha,\beta<\lambda^+$ be regular. Let $M\leq N$ and $\nr{M}=\lambda$.
\begin{enumerate}
\item $N$ is \emph{universal over $M$} ($M<_uN$) if $M<N$ and for any $\nr{N'}=\nr{N}$, there is $f:N'\xrightarrow[M]{}N$.
\item $N$ is \emph{$(\lambda,\alpha)$-limit over $M$} if $\nr{N}=\lambda$ and there exists $\langle M_i:i\leq\alpha\rangle\subseteq K_\lambda$ increasing and continuous such that $M_0=M$, $M_\alpha=N$ and $M_{i+1}$ is universal over $M_i$ for $i<\alpha$. We call $\alpha$ the \emph{length} of $N$.
\item $N$ is \emph{$(\lambda,\alpha)$-limit} if there exists $\nr{M'}=\lambda$ such that $N$ is $(\lambda,\alpha)$-limit over $M'$. 
\item $N$ is \emph{$(\lambda,\geq\beta)$-limit} (\emph{over $M$}) if there exists $\alpha\geq\beta$ such that (2) (resp. (3)) holds.
\item $N$ is \emph{$(\lambda,\lambda^+)$-limit (over $M$)} if $\nr{N}=\lambda^+$ and we replace $\alpha$ by $\lambda^+$ in (2) (resp. (3)).
\item Let $\lambda_1\leq\lambda_2$, then $N$ is \emph{$([\lambda_1,\lambda_2],\geq\beta)$-limit} (\emph{over $M$}) if there exists $\lambda\in[\lambda_1,\lambda_2]$ such that $N$ is {$(\lambda,\geq\beta)$-limit} ({over $M$}).
\item If $\lambda>\ls$, we say $M$ is \emph{$\lambda$-saturated} if for any $M'\leq M$, $\nr{M'}<\lambda$, $M\vDash\gs(M')$.
\item $M$ is \emph{saturated} if it is $\nr{M}$-saturated.
\end{enumerate}
\end{definition}
In general, we do not know limit models or saturated models are closed under chains, so they do not necessary form an AEC. We adapt \cite[Definition 5.3]{s6} to capture such behaviours. 
\begin{definition}\mylabel{skedef}{Definition \thetheorem}
An abstract class ${\bf K_1}$ is a \emph{$\mu$-skeleton} of ${\bf K}$ if the following is satisfied:
\begin{enumerate}
\item ${\bf K_1}$ is a sub-AC of ${\bf K_\mu}$: $K_1\subseteq K_\mu$ and for any $M,N\in K_1$, $M\lkk N$ implies $M\lk N$. 
\item For any $M\in K_\mu$, there is $M'\in K_1$ such that $M\lk M'$.
\item Let $\alpha$ be an ordinal and $\langle M_i:i<\alpha\rangle$ be $\lk$-increasing in $K_1$. There exists $N\in K_1$ such that for all $i<\alpha$, $M_i\lkk N$ (the original definition requires strict inequality but it is immaterial under $NMM$). 
\end{enumerate}
We say ${\bf K_1}$ is a \emph{$(\geq\mu)$-skeleton} of ${\bf K}$ if the above items hold for $K_{\geq\mu}$ in place of $K_\mu$.
\end{definition}
By \cite[II Claim 1.16]{shh}%or \ref{} for a similar proof
, limit models in $\mu$ with $\lk$ form a $\mu$-skeleton of ${\bf K}$. Similarly let $\alpha<\mu^+$ be regular, then $(\geq\mu,\geq\alpha)$-limits form a $(\geq\mu)$-skeleton of ${\bf K}$. 

On the other hand, good frames were developed by Shelah \cite{shh} for AECs in a range of cardinals. \cite{s6} defined good frames over a coherent abstract class. We specialize the abstract class to a skeleton of an AEC. 
\begin{definition}
Let ${\bf K}$ be an AEC and ${\bf K_1}$ be a $\mu$-skeleton of ${\bf K}$. We say a nonforking relation is a \emph{good $\mu$-frame over the skeleton of ${\bf K_1}$} if the following holds:
\begin{enumerate}
\item The nonforking relation is a binary relation between a type $p\in\gs(N)$ and a model $M\leq_{\bf K_1}N$. We say $p$ does not fork over $M$ if the relation holds between $p$ and $M$. Otherwise we say $p$ forks over $M$. 
\item Invariance: if $f\in\oop{Aut}(\mn)$ and $p$ does not fork over $M$, then $f(p)$ does not fork over $f[M]$.
\item Monotonicity: if $p\in\gs(N)$ does not fork over $M$ and $M\leq_{\bf K_1}M'\leq_{\bf K_1}N$ for some $M'\in K_1$, then $p\restriction M'$ does not fork over $M$ while $p$ itself does not fork over $M'$.
\item Existence: if $M\in K_1$ and $p\in\gs(M)$, then $p$ does not fork over $M$.
\item Extension: if $M\leq_{\bf K_1}N\leq_{\bf K_1}N'$ and $p\in\gs(N)$ does not fork over $M$, then there is $q\in\gs(N')$ such that $q\supseteq p$ and $q$ does not fork over $M$.
\item Uniqueness: if $p,q\in\gs(N)$ do not fork over $M$ and $p\restriction M=q\restriction M$, then $p=q$.
\item Transitivity: if $M_0\leq_{\bf K_1}M_1\leq_{\bf K_1}M_2$, $p\in\gs(M_2)$ does not fork over $M_1$, $p\restriction M_1$ does not fork over $M_0$, then $p$ does not fork over $M_0$. 
\item Local character $\al$: if $\delta$ is an ordinal of cofinality $\geq\al$, $\langle M_i:i\leq\delta\rangle$ is $\leq_{\bf K_1}$-increasing and continuous, then there is $i<\delta$ such that $p$ does not fork over $M_i$.
\item Continuity: Let $\delta$ is a limit ordinal and $\langle M_i:i\leq\delta\rangle$ be $\leq_{\bf K_1}$-increasing and continuous. If for all $1\leq i<\delta$, $p_i\in\gs(M_i)$ does not fork over $M_0$ and $p_{i+1}\supseteq p_i$, then $p_\delta$ does not fork over $M_0$.
\item Symmetry: let $M\leq_{\bf K_1}N$, $b\in |N|$, $\gtp(b/M)$ do not fork over $M$, $\gtp(a/N)$ do not fork over $M$. There is $N_a\geq_{\bf K_1}M$ such that $\gtp(b/N_a)$ do not fork over $M$. 
\end{enumerate}
If the above holds for a $(\geq\mu)$-skeleton ${\bf K_1}$, then we say the nonforking relation is a \emph{good $(\geq\mu)$-frame over the skeleton ${\bf K_1}$}.
If ${\bf K_1}$ is itself an AEC (in $\mu$), then we omit ``skeleton''. Let $\alpha<\mu^+$ be regular. We say a nonforking relation has local character $\alpha$ if we replace ``$\al$'' in item (8) by $\alpha$.
\end{definition}
\begin{remark}
\begin{enumerate}
\item In this paper, ${\bf K_1}$ will be the $(\mu,\geq\alpha)$-limit models for some $\alpha<\mu^+$, with $\leq_{\bf K_1}=\leq_u$ (the latter is in ${\bf K})$. 
\item In \ref{BGvarfact}, we will draw results of a good frame over longer types, where we allow the types in the above definition to be of arbitrary length. Extension property will have an extra clause that allows extension of a shorter type to a longer one that still does not fork over the same base.
\item Some of the properties of a good frame imply or simply one another. Instead of using a minimalistic formulation (for example in \cite[Definition 17.1]{snote}), we keep all the properties because sometimes it is easier to deduce a certain property first.
\end{enumerate}
\end{remark}
\section{Properties of nonsplitting}\label{ufsec3}

Let $p\in\gs(N)$, $f:N\rightarrow N'$, we write $f(p)\defeq\gtp(f^+(d)/f(N))$ where $f^+$ extends $f$ to include some $d\vDash p$ in its domain. 
\begin{proposition}\mylabel{gsplnotice}{Proposition \thetheorem}
Such $f^+$ exists by $AP$ and $f(p)$ is independent of the choice of $f^+$. 
\end{proposition}
\begin{proof}
Pick $a\in N_1\geq$ realizing $p$, use $AP$ to obtain $f_1^+:a\mapsto c$ extending $f$ (enlarge $N_1$ if necessary so that $f_1^+(N_1)$ contains $f(N)$). 

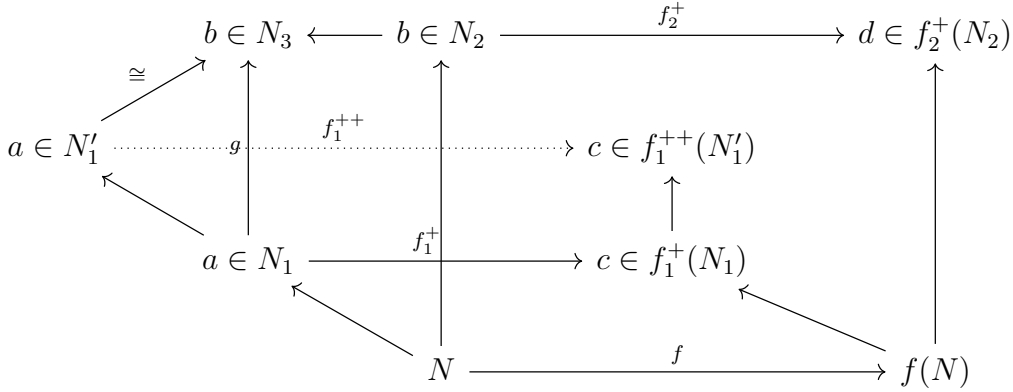
\begin{center}
\begin{tikzcd}
                                                              & b\in N_3                                                & b\in N_2 \arrow[rr, "f_2^+"] \arrow[l]   &                           & d\in f^+_2(N_2)             \\
a\in N_1' \arrow[ru, "\cong"] \arrow[rrr, "f_1^{++}", dotted] &                                                         &                                          & c\in f_1^{++}(N_1')       &                             \\
                                                              & a\in N_1 \arrow[rr, "f_1^+\quad\;"] \arrow[uu, "g"] \arrow[lu] &                                          & c\in f^+_1(N_1) \arrow[u] &                             \\
                                                              &                                                         & N \arrow[lu] \arrow[uuu] \arrow[rr, "f"] &                           & f(N) \arrow[uuu] \arrow[lu]
\end{tikzcd}
\end{center}
Suppose $b\in N_2$ realizes $p$ and there is $f_2^+:b\mapsto d$ extending $f$. Extend $N_2$ so that $f_2^+$ is an isomorphism. We need to find $h:d\mapsto c$ which fixes $f(N)$. Since $a,b\vDash p$, by $AP$ there is $N_3\ni b$ and $g:N_1\xrightarrow[N]{}N_3$ that maps $a$ to $b$. Extend $g$ to an isomorphism $N_1'\cong_N N_3\geq N_2$. By $AP$ again, obtain $f_1^{++}$ of domain $N_1'$ extending $f_1^+$. Therefore, $d\in f(N_2^+)$ and $f_1^{++}\circ g^{-1}\circ\oop{id}_{N_2}\circ(f_2^+)^{-1}(d)=c$. Hence we can take $h\defeq f_1^{++}\circ g^{-1}\circ\oop{id}_{N_2}\circ(f_2^+)^{-1}:f_2^+(N_2)\xrightarrow[f(N)]{}f_1^{++}(N_1') $.
\end{proof}
\begin{definition}\mylabel{spldef}{Definition \thetheorem}
Let $M,N\in K$, $p\in\gs(N)$. $p$ \emph{$\mu$-splits over $M$} if there exists $N_1,N_2$ of size $\mu$ such that $M\leq N_1,N_2\leq N$ and $f:N_1\cong_M{}N_2$ such that $f(p)\restriction N_2\neq p\restriction N_2$. 
\end{definition}
\begin{proposition}[Monotonicity of nonsplitting]\mylabel{monspl}{Proposition \thetheorem}
Let $M,N\in K_\mu$, $p\in\gs(N)$ do not $\mu$-split over $M$. For any $M_1,N_1$ with $M\leq M_1\leq N_1\leq N$, we have $p\restriction N_1$ does not $\mu$-split over $M_1$.
\end{proposition}

\begin{proposition}\mylabel{tamespl}{Proposition \thetheorem}
Let $M,N\in K$, $M\in K_\mu$ and $p\in\gs(N)$. $p$ $\mu$-splits over $M$ iff $p$ $(\geq\mu)$-splits over $M$ (the witnesses $N_1,N_2$ can be in $K_{\geq\mu}$).
\end{proposition}
\begin{proof}
We sketch the backward direction: pick $N_1,N_2\in K_{\geq\mu}$ witnessing $p$ $(\geq\mu)$-splits over $M$. By $\mu$-tameness and L\"{o}wenheim-Skolem axiom, we may assume $N_1,N_2\in K_\mu$. 
\end{proof}
\begin{definition}\mylabel{condef}{Definition \thetheorem}
Let $\chi$ be a regular cardinal.
\begin{enumerate}
\item A chain $\langle M_i:i\leq\delta\rangle$ is \emph{u-increasing} if $M_{i+1}>_uM_i$ for all $i<\delta$.
\item ${\bf K}$ satisfies \emph{continuity of $\mu$-nonsplitting} if for any limit ordinal $\delta$, $\langle M_i:i\leq\delta\rangle\subseteq K_\mu$ u-increasing and continuous, $p\in\gs(M_\delta)$, 
$$p\restriction M_i\text{ does not $\mu$-split over $M_0$ for $i<\delta$}\Rightarrow p\text{ does not $\mu$-split over $M_0$}.$$
\item ${\bf K}$ has \emph{$\chi$-weak local character of $\mu$-nonsplitting} if for any limit ordinal $\delta\geq\chi$, $\langle M_i:i\leq\delta\rangle\subseteq K_\mu$ u-increasing and continuous, $p\in\gs(M_\delta)$, there is $i<\delta$ such that $p\restriction M_{i+1}$ does not $\mu$-split over $M_i$.
\item ${\bf K}$ has \emph{$\chi$-local character of $\mu$-nonsplitting} if the conclusion in (3) becomes: $p$ does not $\mu$-split over $M_i$. 
\end{enumerate}
We call any $\delta$ that satisfies (3) or (4) a \emph{(weak) local character cardinal}.
\end{definition}
\begin{remark}
When defining the continuity of nonsplitting, we can weaken the statement by removing the assumption that $p$ exists and replacing $p\restriction M_i$ by $p_i$ increasing. This is because we can use \cite[Proposition 5.2]{bonext} to recover $p$. In details, we can use the weaker version of continuity and weak uniqueness (\ref{eeuprop}) to argue that the $p_i$'s form a coherent sequence. $p$ can be defined as the direct limit of the $p_i$'s.
\end{remark}
The following lemma connects the three properties of $\mu$-nonsplitting:
\begin{lemma}\emph{\cite[Lemma 11(1)]{BGVV}}\mylabel{bgvvlem}{Lemma \thetheorem}
If $\mu$ is regular, ${\bf K}$ satisfies continuity of $\mu$-nonsplitting and has $\chi$-weak local character of $\mu$-nonsplitting, then it has $\chi$-local character of $\mu$-nonsplitting.
\end{lemma}
\begin{proof}
Let $\delta$ be a limit ordinal of cofinality $\geq\chi$, $\langle M_i:i\leq\delta\rangle$ u-increasing and continuous. Suppose $p\in\gs(M_\delta)$ splits over $M_i$ for all $i<\delta$. Define $i_0\defeq0$. By $\delta$ regular and continuity of $\mu$-nonsplitting, build an increasing and continuous sequence of indices $\langle i_k:k<\delta\rangle$ such that $p\restriction M_{i_{k+1}}$ $\mu$-splits over $M_{i_k}$. Notice that $M_{i_{k+1}}>_uM_{i_k}$. Then applying $\chi$-weak local character to $\langle M_{i_k}:k<\delta\rangle$ yields a contradiction.
\end{proof}
From stability (even without continuity of nonsplitting), it is always possible to obtain weak local character of nonsplitting. Shelah sketched the proof and alluded to the first-order analog, so we give details here. %set version in my
\begin{lemma}\emph{\cite[Claim 3.3(2)]{sh394}}\mylabel{394lem}{Lemma \thetheorem}
If ${\bf K}$ is stable in $\mu$ (which is in \ref{assum1}), then for some $\chi\leq\mu$, it has weak $\chi$-local character of $\mu$-nonsplitting.
\end{lemma}
\begin{proof}
Pick $\chi\leq\mu$ minimum such that $2^\chi>\mu$. Suppose we have $\langle M_i:i\leq\chi\rangle$ u-increasing and continuous and $d\vDash p\in\gs(M_\chi)$ such that for all $i<\chi$, $p\restriction M_{i+1}$ $\mu$-splits over $p\restriction M_i$. Then for $i<\chi$, we have $N_i^1$ and $N_i^2$ of size $\mu$, $M_i\leq N_i^1,N_i^2\leq M_{i+1}$, $f_i:N_i^1\cong_{M_i}N_i^2$ and $f_i(p)\restriction N_i^2\neq p\restriction N_i^2$. We build $\langle M_i':i\leq\chi \rangle$ and $\langle h_\eta:M_{l(\eta)}\xrightarrow[M_0]{}M_{l(\eta)}'\,|\,\eta\in2^{\leq\chi}\rangle$ both increasing and continuous with the following requirements:
\begin{enumerate}
\item $h_{\langle\rangle}\defeq\oop{id}_{M_0}$ and $M_0'\defeq M_0$.
\item For $\eta\in2^{<\chi}$, $h_{\eta^\frown0}\restriction N_{l(\eta)}^2=h_{\eta^\frown1}\restriction N_{l(\eta)}^2$.
\end{enumerate}
\begin{center}
\begin{tikzcd}
                                                         & M_{hf_i} \arrow[rr, "g_1"]                     &  & M_{i+1}'                   \\
M_{i+1} \arrow[ru, "g_0"] \arrow[rrr, "h_{\nu^\frown0}"] &                                              &  & M^{**} \arrow[u]           \\
N_i^1 \arrow[r, "f_i","\cong"'] \arrow[u]                           & N_i^2 \arrow[rr, "h"] \arrow[lu]             &  & M^* \arrow[lluu] \arrow[u] \\
                                                         & M_i \arrow[rr, "h_\nu"] \arrow[lu] \arrow[u] &  & M_i' \arrow[u]            
\end{tikzcd}
\end{center}

We specify the successor step: suppose $l(\nu)=i$ and $h_\nu$ has been constructed. By $AP$, obtain
\begin{enumerate}
\item $h:N_i^2\rightarrow M^*\geq M_i'$ with $h\supseteq h_\nu$.
\item $h_{\nu^\frown0}:M_{i+1}\rightarrow M^{**}\geq M^*$ with $h_{\nu^\frown0}\supseteq h$.
\item $g_0:M_{i+1}\rightarrow M_{hf_i}\geq M^*$ with $g_0\supseteq h\circ f_i$. 
\item $g_1:M_{hf_i}\rightarrow M_{i+1}'\geq M^{**}$ with $g_1\circ g_0=h_{\nu^\frown0}$.
\end{enumerate}
Define $h_{\nu^\frown1}\defeq g_1\circ g_0:M_{i+1}\rightarrow M_{i+1}'$. By diagram chasing, $h_{\nu^\frown1}\restriction M_i=g_1\circ g_0\restriction M_i=g_1\circ h\circ f_i\restriction M_i=g_1\circ h\restriction M_i=h\restriction M_i=h_\nu\restriction M_i$. On the other hand, $h_{\nu^\frown0}\restriction M_i=h\restriction M_i=h_\nu\restriction M_i$. Therefore the maps are increasing. Now $h_{\nu^\frown1}\restriction N_i^2=g_1\circ g_0\restriction N_i^2=h_{\nu^\frown0}\restriction N_i^2$ by item (4) in our construction. 

For $\eta\in2^\chi$, extend $h_\eta$ so that its range includes $M_\chi'$ and its domain includes $d$. We show that $\{\gtp(h_\eta(d)/M_{\chi}'):\eta\in2^\chi\}$ are pairwise distinct. For any $\eta\neq\nu\in2^\chi$, pick the minimum $i<\chi$ such that $\eta[i]\neq\nu[i]$. Without loss of generality, assume $\eta[i]=0$, $\nu[i]=1$. Using the diagram above (see the comment before \ref{gsplnotice}),
\begin{align*}
\gtp(h_\eta(d)/M_{\chi}')&\supseteq\gtp(h_{\eta}(d)/h(N_i^2))\\
&=h(\gtp(d/N_i^2))\\
&\neq h\circ f_i(\gtp(d/N_i^1))\\
&=g_1\circ h\circ f_i(\gtp(d/N_i^1))\\
&\subseteq\gtp(h_\nu(d)/M_{\chi}')
\end{align*}
This contradicts the stability in $\mu$.
\end{proof}
\begin{proposition}\mylabel{394prop}{Proposition \thetheorem}
If $\mu$ is regular, then for some $\chi\leq\mu$, ${\bf K}$ has the $\chi$-local character of $\mu$-nonsplitting.
\end{proposition}
\begin{proof}
By \ref{394lem}, ${\bf K}$ has $\mu$-weak local character of $\mu$-nonsplitting. By \ref{bgvvlem} (together with continuity of $\mu$-nonsplitting in \ref{assum1}), ${\bf K}$ has $\mu$-local character of $\mu$-nonsplitting. Hence $\chi$ exists and $\chi\leq\mu$.
\end{proof}
From now on, we fix
\begin{definition}\mylabel{chi}{Definition \thetheorem}
$\chi$ is the minimum local character cardinal of $\mu$-nonsplitting. $\chi\leq\mu$ if either $\mu$ is regular (by the previous proposition), or $\mu$ is greater than some regular stability cardinal $\xi$ where ${\bf K}$ has continuity of $\xi$-nonsplitting and is $\xi$-tame (by \ref{chitrans}). 
\end{definition}
\begin{remark}
Without continuity of nonsplitting, it is not clear whether there can be gaps between the local character cardinals: \ref{condef}(4) might hold for $\delta=\aleph_0$ and $\delta=\aleph_2$ but not $\delta=\aleph_1$. In that case defining $\chi$ as the \emph{minimum} local character cardinal might not be useful. Similar obstacles form when we only know a particular $\lambda$ is a local character cardinal but not necessary those above $\lambda$. 

Meanwhile, weak local character cardinals close upwards and we can eliminate the above situation by assuming continuity of nonsplitting: if we know $\chi$ is the minimum local character cardinal, then it is also a weak local character cardinal, so are all regular cardinals between $[\chi,\mu^+)$. By the proof of  \ref{bgvvlem}, the regular cardinals between $[\chi,\mu^+)$ are all local character cardinals.
\end{remark}

We now state the existence, extension, weak uniqueness and weak transitivity properties of $\mu$-nonsplitting. The original proof for weak uniqueness assumes $\nr{M}=\mu$ but it is not necessary; while that for extension and for weak transitivity assume all models are in $K_\mu$; but under tameness we can just require $\nr{M}=\nr{N}$.

\begin{proposition}\mylabel{eeuprop}{Proposition \thetheorem}
Let $M_0<_u M\leq N$ where $\nr{M_0}=\mu$.
\begin{enumerate}
\item \emph{\cite[Claim 3.3(1)]{sh394}} (Existence) If $p\in\gs(N)$, there is $N_0\leq N$ of size $\mu$ such that $p$ does not $\mu$-split over $N_0$.
\item \emph{\cite[Theorem 6.2]{GV06b}} (Weak uniqueness) If $p,q\in\gs(N)$ both do not $\mu$-split over $M_0$, and $p\restriction M=q\restriction M$, then $p=q$.
\item \emph{\cite[Theorem 6.1]{GV06b}} (Extension) Suppose $\nr{M}=\nr{N}$. For any $p\in\gs(M)$ that does not $\mu$-split over $M_0$, there is $q\in\gs(N)$ extending $p$ such that $q$ does not $\mu$-split over $M_0$.
\item \emph{\cite[Proposition 3.7]{s3}} (Weak transitivity) Suppose $\nr{M}=\nr{N}$. Let $M^*\leq M_0$ and $p\in\gs(N)$. If $p$ does not $\mu$-split over $M_0$ while $p\restriction M$ does not $\mu$-split over $M^*$, then $p$ does not $\mu$-split over $M^*$. 
\end{enumerate}
\end{proposition}
\begin{proof}
\begin{enumerate}
\item We skip the proof, which has the same spirit as that of \ref{394lem}. %syntactic
\item By stability in $\mu$%build universal
, we may assume that $\nr{M}=\mu$. Suppose $p\neq q$, by tameness in $\mu$ we may find $M'\in K_\mu$ such that $M\leq M'\leq N$ and $p\restriction M'\neq q\restriction M'$. By $M_0<_uM$ and $M_0<N$, we can find $f:M'\xrightarrow[M_0]{}M$. Using nonsplitting twice, we have $p\restriction f(M')=f(p)$ and $q\restriction f(M')=f(q)$. But $f(M')\leq M$ implies $p\restriction f(M')=q\restriction f(M')$. Hence $f(p)=f(q)$ and $p=q$. 
\item By universality of $M$, find $f:N\xrightarrow[M_0]{}M$. We can set $q\defeq f^{-1}(p\restriction f(N))$.
\item Let $q\defeq p\restriction M$. By extension, obtain $q'\supseteq q$ in $\gs(N)$ such that $q'$ does not $\mu$-split over $M^*$. Now $p\restriction M=q\restriction M=q'\restriction M$ and both $p,q'$ do not $\mu$-split over $M_0$ (for $q'$ use monotonicity, see \ref{monspl}). By weak uniqueness, $p=q'$ and the latter does not $\mu$-split over $M^*$. 
\end{enumerate}
\end{proof}

Transitivity does not hold in general for $\mu$-nonsplitting. The following example is sketched in \cite[Example 19.3]{bal}.

\begin{example}\mylabel{balex}{Example \thetheorem}
Let $T$ be the first-order theory of a single equivalence relation $E$ with infinitely many equivalence classes and each class is infinite. Let $M\leq N$ where $N$ contains (representatives of) two more classes than $M$. Let $d$ be an element. Then $\tp(d/N)$ splits over $M$ iff $dEa$ for some element $a\in N$ but $\neg dEb$ for any $b\in M$. Meanwhile, suppose $M_0\leq M$ both of size $\mu$, then $M_0<_u M$ iff $M$ contains $\mu$-many new classes and each class extends $\mu$ many elements. Now require $M_0<_u M$ while $N$ contains only an extra class than $M$, say witnessed by $d$, then $\tp(d/N)$ cannot split over $M$. Also $\tp(d/M)$ does not split over $M_0$ because $d$ is not equivalent to any elements from $M$. Finally $\tp(d/N)$ splits over $M_0$ because it contains two more classes than $M_0$ (one must be from $M$). 
\end{example}
The same argument does not work if also $M<_u N$ because $N$ would contain two more classes than $M$ and they will witness $\tp(d/N)$ splits over $M$. Baldwin originally assigned it as \cite[Exercise 12.9]{bal} but later \cite{baler} retracted the claim. %We will show in \ref{balconj} that it is true under continuity of $\mu$-nonsplitting.
\begin{ques}
When models are ordered by $\leq_u$, 
\begin{enumerate}
\item does uniqueness of $\mu$-nonsplitting hold? Namely, let $M<_uN$ both in $K_\mu$, $p,q\in\gs(N)$ both do not $\mu$-split over $M$, $p\restriction M=q\restriction M$, then $p=q$.
\item does transitivity of $\mu$-nonsplitting hold? Namely, let $M_0<_uM<_uN$ all in $K_\mu$, $p\in\gs(N)$ does not $\mu$-split over $M$ and $p\restriction M$ does not $\mu$-split over $M_0$, then $p$ does not $\mu$-split over $M_0$.
\end{enumerate}
\end{ques}

In \ref{assum1}, we assumed continuity of $\mu$-nonsplitting. One way to obtain it is to assume superstability which is stronger. Another way is to assume $\omega$-type locality.

\begin{definition}
\begin{enumerate}
\item \cite[Definition 7.12]{Gclass} Let $\lambda\geq\ls$, ${\bf K}$ is \emph{$\lambda$-superstable} if it is stable in $\lambda$ and has $\al$-local character of $\lambda$-nonsplitting.
\item \cite[Definition 11.4]{bal} Types in ${\bf K}$ are \emph{$\omega$-local} if: for any limit ordinal $\alpha$, $\langle M_i:i\leq\alpha\rangle$ increasing and continuous, $p,q\in\gs(M_\alpha)$ and $p\restriction M_i=q\restriction M_i$ for all $i<\alpha$, then $p=q$.
\end{enumerate}
\end{definition}
\begin{proposition}\mylabel{weaker}{Proposition \thetheorem}
Let ${\bf K}$ satisfy \ref{assum1} except for the continuity of $\mu$-nonsplitting. It will satisfy the continuity of $\mu$-nonsplitting if either
\begin{enumerate}
\item ${\bf K}$ is $\mu$-superstable; or%\emph{\cite[Lemma 3.12(2)]{s3}}
\item Types in ${\bf K}$ are $\omega$-local.
\end{enumerate}
\end{proposition}
\begin{proof}
For (1), it suffices to prove that for any regular $\lambda\geq\al$, $\lambda$-local character implies continuity of $\mu$-nonsplitting over chains of cofinality $\geq\lambda$. Let $\langle M_i:i\leq\lambda\rangle$ be u-increasing and continuous. Suppose $p\in\gs(M_\lambda)$ satisfies $p\restriction M_i$ does not $\mu$-split over $M_0$ for all $i<\lambda$. By $\lambda$-local character, $p$ does not $\mu$-split over some $M_i$. If $i=0$ we are done. Otherwise, we have $M_0<_uM_i<_uM_{i+1}<_uM_\lambda$. By assmption, $p\restriction M_{i+1}$ does not $\mu$-split over $M_0$. By weak transitivity (\ref{eeuprop}), $p$ does not $\mu$-split over $M_0$ as desired.

For (2), let $\langle M_i:i\leq\lambda\rangle$ and $p$ as above. By assumption $p\restriction M_1$ does not $\mu$-split over $M_0$ and $M_1>_uM_0$. By extension (\ref{eeuprop}), there is $q\supseteq p\restriction M_1$ in $\gs(M_\lambda)$ such that $q$ does not $\mu$-split over $M_0$. By monotonicity, for $2\leq i<\lambda$, $q\restriction M_i$ does not $\mu$-split over $M_0$. Now $(q\restriction M_i)\restriction M_1=p\restriction M_1=(p\restriction M_i)\restriction M_1$, we can use weak uniqueness (\ref{eeuprop}) to inductively show that $q\restriction M_i=p\restriction M_i$ for all $i<\lambda$. By $\omega$-locality, $p=q$ and the latter does not $\mu$-split over $M_0$ as desired.
\end{proof}

Once we have continuity of $\mu$-nonsplitting in $K_\mu$, it automatically works for $K_{\geq\mu}$:
\begin{proposition}\mylabel{congeqmu}{Proposition \thetheorem}
Let $\delta$ be a limit ordinal, $\langle M_i:i\leq\delta\rangle\subseteq K_{\geq\mu}$ be u-increasing and continuous, $p\in\gs(M_\delta)$. If for all $i<\delta$, $p\restriction M_i$ does not $\mu$-split over $M_0$, then $p$ also does not $\mu$-split over $M_0$.
\end{proposition}
\begin{proof}
The statement is vacuous when $M_0\in K_{>\mu}$ so we assume $M_0\in K_\mu$. By cofinality argument we may also assume $\cf(\delta)\leq\mu$. Suppose $p$ $\mu$-splits over $M_0$ and pick witnesses $N^a$ and $N^b$ of size $\mu$. Using stability, define another u-increasing and continuous chain $\langle N_i:i\leq\delta\rangle\subseteq K_\mu$ such that:
\begin{enumerate}
\item For $i\leq\delta$, $N_i\leq M_i$. 
\item $N_\delta$ contains $N^a$ and $N^b$.
\item $N_0\defeq M_0$.
\item For $i\leq\delta$, $|N_i|\supseteq |M_i|\cap(|N^a|\cup|N^b|)$.
\end{enumerate}
By assumption each $p\restriction M_i$ does not $\mu$-split over $M_0$, so by monotonicity $p\restriction N_i$ does not $\mu$-split over $N_0=M_0$. By continuity of $\mu$-nonsplitting, $p\restriction N_\delta$ does not $\mu$-split over $N_0$, contradicting item (2) above. 
\end{proof}

\section{Good frame over $(\geq\chi)$-limit models except symmetry}\label{ufsec4}

As seen in \ref{eeuprop}, $\mu$-nonsplitting only satisfies weak transitivity but not transitivity, which is a key property of a good frame. We will adapt \cite[Definitions 3.8, 4.2]{s3} to define nonforking from nonsplitting to solve this problem.

\begin{definition}\mylabel{nfdef}{Definition \thetheorem}
Let $M\leq N$ in $K_{\geq\mu}$ and $p\in\gs(N)$.
\begin{enumerate}
\item \emph{$p$ (explicitly) does not $\mu$-fork over $(M_0,M)$} if $M_0\in K_\mu$, $M_0<_uM$ and $p$ does not $\mu$-split over $M_0$.
\item \emph{$p$ does not $\mu$-fork over $M$} if there exists $M_0$ satisfying (1).
\end{enumerate}
We call $M_0$ the \emph{witness} to $\mu$-nonforking over $M$.
\end{definition}

The main difficulty of the above definition is that different $\mu$-nonforkings over $M$ may have different witnesses. For extension, the original approach in \cite{s3} was to work in $\mu^+$-saturated models. Later \cite[Proposition 5.1]{vv} replaced it by superstability in an interval, which works for $K_{\geq\mu}$. We weaken the assumption to stability in an interval and continuity of $\mu$-nonsplitting, and use a direct limit argument similar to that of \cite[Theorem 5.3]{bonext}.

\begin{proposition}[Extension]\mylabel{gfe}{Proposition \thetheorem}
Let $M\leq N\leq N'$ in $K_{\geq\mu}$. If ${\bf K}$ is stable in $[\nr{N},\nr{N'}]$ and $p\in\gs(N)$ does not $\mu$-fork over $M$, then there is $q\supseteq p$ in $\gs(N')$ such that $q$ does not $\mu$-fork over $M$.
\end{proposition}
\begin{proof}
Since $p$ does not $\mu$-fork over $M$, we can find witness $M_0\in K_\mu$ such that $M_0<_uM$ and $p$ does not $\mu$-split over $M_0$. If $\nr{N}=\nr{N'}$, we can use extension of nonsplitting (\ref{eeuprop}) to obtain (the unique) $q\in\gs(N')$ extending $p$ which does not $\mu$-split over $M_0$. By definition $q$ does not $\mu$-fork over $M$.

If $\nr{N}<\nr{N'}$, first we assume $N'=\bigcup\{N_i:i\leq\alpha\}$ u-increasing and continuous where $N_0=N$, $N_{\alpha}=N'$ for some $\alpha$. We will define a coherent sequence $\langle p_i:i\leq\alpha\rangle$ such that $p_i$ is a nonsplitting extension of $p$ in $\gs(N_i)$. The first paragraph gives the successor step. For limit step $\delta\leq\alpha$, we take the direct limit %ch1
to obtain an extension $p_\delta$ of $\langle p_i:i<\delta\rangle$. Since all previous $p_i$ does not $\mu$-split over $M_0$, by \ref{congeqmu}, $p_\delta$ also does not $\mu$-split over $M_0$. After the construction has finished, we obtain $q\defeq p_{\alpha}$ a nonsplitting extension of $p$ in $\gs(N')$. Since $M_0<_uM\leq N'$, we still have $q$ does not $\mu$-fork over $M$. 

In the general case where $N\leq N'$, extend $N'\leq N''$ so that $\nr{N''}=\nr{N'}$ and $N''$ contains a limit model over $N$ of size $\nr{N'}$. The construction is possible by stability in $[\nr{N},\nr{N'}]$. Then we can extend $p$ to a nonforking $q''\in\gs(N'')$ and use monotonicity to obtain the desired $q$.
\end{proof}
\begin{corollary}\mylabel{extnsp}{Corollary \thetheorem}
Let $M_0<_u M\leq N'$ with $M_0\in K_\mu$. If ${\bf K}$ is stable in $[\nr{M},\nr{N'}]$ and $p\in\gs(M)$ does not $\mu$-split over $M_0$, then there is $q\supseteq p$ in $\gs(N')$ such that $q$ does not $\mu$-split over $M_0$.
\end{corollary}
\begin{proof}
Run through the exact same proof as in \ref{gfe}, where $M=N$ and $M_0$ is given in the hypothesis.
\end{proof}
For continuity, the original approach in \cite[Lemma 4.12]{s3} was to deduce it from superstability (which we do not assume) and transitivity. Transitivity there was obtained from extension and uniqueness, and uniqueness was proved in \cite[Lemma 5.3]{s3} for $\mu^+$-saturated models only (or assuming superstability in \cite[Lemma 2.12]{s22}). Our new argument uses weak transitivity and continuity of $\mu$-nonsplitting to show that continuity of $\mu$-nonforking holds for a universally increasing chain in $K_\mu$. The case in $K_{\geq\mu}$ will be proved after we have developed transitivity and local character of nonforking.

\begin{proposition}[Continuity 1]\mylabel{gfc}{Proposition \thetheorem}
Let $\delta<\mu^+$ be a limit ordinal and $\langle M_i:i\leq\delta\rangle\subseteq K_{\mu}$ be u-increasing and continuous. Let $p\in\gs(M_\delta)$ satisfy $p\restriction M_i$ does not $\mu$-fork over $M_0$ for all $1\leq i<\delta$. Then $p$ also does not $\mu$-fork over $M_0$.
\end{proposition}
\begin{proof}
For $1\leq i<\delta$, since $p\restriction M_i$ does not $\mu$-fork over $M_0$, we can find $M^i<_uM_0$ of size $\mu$ such that $p\restriction M_i$ does not $\mu$-split over $M^i$. By monotonicity of nonsplitting, $p\restriction M_i$ does not $\mu$-split over $M_0$. By continuity of $\mu$-nonsplitting, $p$ does not $\mu$-split over $M_0$. Since $M^1<_u M_0<_u M_1<_u M_\delta$, by weak transitivity (\ref{eeuprop}) $p$ does not $\mu$-split over $M^1$. (By a similar argument, it does not $\mu$-split over other $M^i$.) By definition $p$ does not $\mu$-fork over $M_0$. 
\end{proof}

We now show uniqueness of nonforking in $K_\mu$, by generalizing the argument in \cite{s22}. Instead of superstability, we stick to our \ref{assum1}. Fact 2.9 in that paper will be replaced by our \ref{gfe}. The requirement that $M_0,M_1$ be limit models is removed.

\begin{proposition}[Uniqueness 1]\mylabel{gfu}{Proposition \thetheorem}
Let $M_0\leq M_1$ in $K_\mu$ and $p_0,p_1\in\gs(M_1)$ both do not $\mu$-fork over $M_0$. If in addition $p_{\langle\rangle}\defeq p_0\restriction M_0=p_1\restriction M_0$, then $p_0=p_1$. 
\end{proposition}
\begin{proof}
Suppose the proposition is false. Let $N_0<_u M_0$ and $N_1<_uM_0$ such that $p_0$ does not $\mu$-split over $N_0$ while $p_1$ does not $\mu$-split over $N_1$ (necessarily $N_0\neq N_1$ by weak uniqueness of nonsplitting). We will build a u-increasing and continuous $\langle M_i:i\leq\mu\rangle\subseteq K_\mu$ and a coherent $\langle p_\eta\in\gs(M_{l(\eta)}):\eta\in2^{\leq\mu}\rangle$ such that for all $\nu\in2^{<\mu}$, $p_{\nu^\frown 0}$ and $p_{\nu^\frown 1}$ are distinct nonforking extensions of $p_{\nu}$. If done $\{p_\eta:\eta\in2^\mu\}$ will contradict stability in $\mu$. 

The base case is given by the assumption. For successor case, suppose $M_i$ and $\{p_\eta:\eta\in2^i\}$ have been constructed for some $1\leq i<\mu$. Define $M_{i+1}'$ to be a $(\mu,\omega)$-limit over $M_i$. Fix $\eta\in2^i$, we will define $p_{\eta^\frown0},p_{\eta^\frown1}\in\gs(M_{i+1}')$ nonforking extensions of $p_\eta$ (nonsplitting will be witnessed by different models; otherwise weak uniqueness of nonsplitting applies). Since $p_\eta$ does not $\mu$-fork over $M_0$, we can find $N_{\eta}<_uM_0$ such that $p_\eta$ does not $\mu$-split over $N_\eta$. Pick $p_\eta^+\in\gs(M_{i+1}')$ a nonsplitting extension of $p_\eta$. On the other hand, obtain $N_\eta'<_uN^*<_uM_0$ such that $N^*$ is a $(\mu,\omega)$-limit over $N_\eta'$ and $N_\eta'>_uN_\eta$. By uniqueness of limit models over $N_\eta$ of the same length, %ch1
there is $f:M_{i+1}'\cong_{N_\eta'}N^*$. 

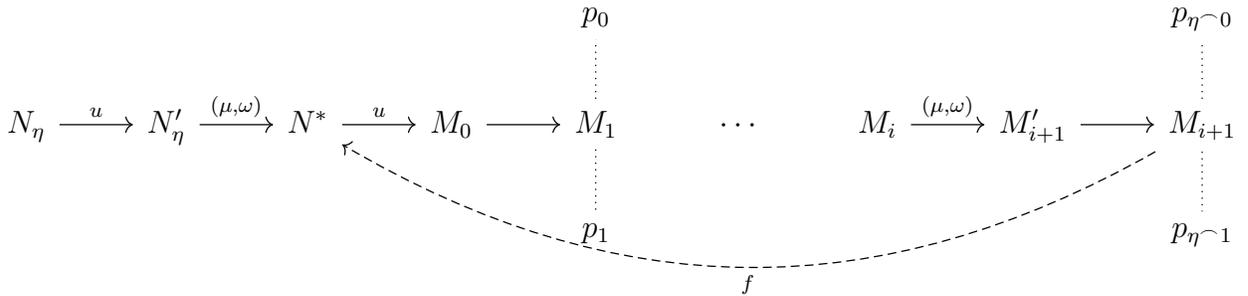
\begin{center}
\begin{tikzcd}
                  &    &                                 &                    & p_0                                                       &        &                                 &                    & p_{\eta^\frown0}                                                                                    \\
N_\eta\arrow[r,"u"]&N_\eta' \arrow[r, "{(\mu,\omega)}"]  & N^*\arrow[r, "u"] & M_0  \arrow[r]&  M_1 \arrow[d, no head, dotted] \arrow[u, no head, dotted] & \cdots & M_i \arrow[r, "{(\mu,\omega)}"] & M_{i+1}' \arrow[r] & M_{i+1} \arrow[u, no head, dotted] \arrow[d, no head, dotted] \arrow[llllll, "f", dashed, bend left] \\
                   &   &                                 &                    & p_1                                                       &        &                                 &                    & p_{\eta^\frown1}                                                                                   
\end{tikzcd}
\end{center}

By invariance of nonsplitting, $f(p_\eta^+)$ does not $\mu$-split over $N_\eta$. By monotonicity of nonsplitting, $p_\eta$, and hence $p_\eta\restriction N^*$ does not $\mu$-split over $N_\eta$. $f(p_\eta^+)\restriction N_\eta'=p_\eta^+\restriction N_\eta'=(p_\eta\restriction N^*)\restriction N_\eta'$. By weak uniqueness of $\mu$-nonsplitting, $f(p_\eta^+)=p_\eta\restriction N^*$. Since $p_\eta\restriction N^*$ has two nonforking extensions $p_0\neq p_1\in\gs(M_1)$ where $M_1>_uN^*$, we can obtain their isomorphic copies $p_{\eta^\frown0}\neq p_{\eta^\frown1}\in\gs(M_{i+1})$ for some $M_{i+1}>_uM_{i+1}'$. They still do not $\mu$-fork over $M_0$ because $M_0$ is fixed (actually $p_{\eta^\frown i}$ does not $\mu$-split over $N_i<_uM_0$). Ensure coherence at the end.

For limit case, let $\eta\in2^\delta$ for some limit ordinal $\delta\leq\mu$. Define $p_\eta\in\gs(M_\delta)$ to be the direct limit of $\langle p_{\eta\restriction i}:i<\delta\rangle$. By \ref{gfc}, $p_\eta$ does not $\mu$-fork over $M_0$. 

\end{proof}
\begin{corollary}[Uniqueness 2]\mylabel{gfuup}{Corollary \thetheorem}
Let $M\leq N$ in $K_{\geq\mu}$ and $p,q\in\gs(N)$ both do not $\mu$-fork over $M$. If in addition $p\restriction M=q\restriction M$, then $p=q$. 
\end{corollary}
\begin{proof}
\ref{gfu} takes care of the case $M,N\in K_{\mu}$. Suppose the corollary is false, then $p\neq q$ and there exist $N^p,N^q<_uM$ such that $p$ does not $\mu$-fork over $N^p$ and $q$ does not $\mu$-fork over $N^q$. We have two cases:
\begin{enumerate}
\item Suppose $M\in K_\mu$ but $N\in K_{>\mu}$. By tameness obtain $N'\in K_\mu$ such that $M\leq N'\leq N$ and $p\restriction N'\neq q\restriction N'$. Together with $p\restriction M=q\restriction M$, it contradicts \ref{gfu}.
\item Suppose $M\in K_{>\mu}$. Obtain $M^p,M^q\leq M$ of size $\mu$ that are universal over $N^p$ and $N^q$ respectively. By L\"{o}wenheim-Skolem axiom, pick $M'\leq M$ of size $\mu$ containing $M^p$ and $M^q$. Thus $M'$ is universal over both $N^p$ and $N^q$, and $p\restriction M'=q\restriction M'$. Since $p\neq q$, tameness gives some $N'\in K_\mu$, $M'\leq N'\leq N$ such that $p\restriction N'\neq q\restriction N'$, which contradicts \ref{gfu}.
\end{enumerate}
\end{proof}
\begin{remark}
The strategy of case (2) cannot be applied to \ref{gfu} because $M'$ might coincide with $M$ and we do not have enough room to invoke weak uniqueness of nonsplitting. This calls for a specific proof in \ref{gfu}. Similarly, we cannot simply invoke weak uniqueness of nonsplitting to prove case (2) because we do not know if $M$ is also universal over $M'$. 
\end{remark}
\begin{corollary}[Transitivity]\mylabel{gft}{Corollary \thetheorem}
Let $M_0\leq M_1\leq M_2$ be in $K_{\geq\mu}$, $p\in\gs(M_2)$. If ${\bf K}$ is stable in $[\nr{M_1},\nr{M_2}]$, $p$ does not $\mu$-fork over $M_1$ and $p\restriction M_1$ does not $\mu$-fork over $M_0$, then $p$ does not $\mu$-fork over $M_0$.
\end{corollary}
\begin{proof}
By \ref{gfe}, obtain $q\supseteq p\restriction M_1$ a nonforking extension in $\gs(M_2)$. Both $q$ and $p$ do not fork over $M_1$ and $q\restriction M_1=p\restriction M_1$. By \ref{gfuup}, $p=q$, but $q$ does not $\mu$-fork over $M_0$. 
\end{proof}
\iffalse
Using the proof of uniqueness of nonforking, we can show uniqueness of nonsplitting.
\begin{corollary}[Uniqueness of nonsplitting]\mylabel{uninsp}{Corollary \thetheorem}
Let $M_0<_uM_1$ in $K_\mu$ and $p_0\neq p_1\in\gs(M_1)$ both do not $\mu$-split over $M_0$. If in addition $p_{\langle\rangle}\defeq p_0\restriction M_0=p_1\restriction M_0$, then $p_0=p_1$. 
\end{corollary}
\begin{proof}
Suppose the corollary is false, $p_0\neq p_1$. As in \ref{gfu}, we will build a u-increasing and continuous $\langle M_i:i\leq\mu\rangle\subseteq K_\mu$ and a coherent $\langle p_\eta\in\gs(M_{l(\eta)}):\eta\in2^{\leq\mu}\rangle$ such that for all $\nu\in2^{<\mu}$, $p_{\nu^\frown 0}$ and $p_{\nu^\frown 1}$ are distinct nonsplitting extensions of $p_{\nu}$.
\end{proof}
We now solve Baldwin's conjecture \cite[Exercise 12.9]{bal} that transitivity of nonsplitting holds for models that are universal over one another.
\begin{corollary}[Transitivity of nonsplitting]\mylabel{balconj}{Corollary \thetheorem}
Let $M_0<_uM_1<_uM_2$ with $M_0\in K_\mu$. If ${\bf K}$ is stable in $[\nr{M_1},\nr{M_2}]$, $p\in\gs(M_2)$ does not $\mu$-split over $M_1$ and $p\restriction M_1$ does not $\mu$-split over $M_0$, then $p$ does not $\mu$-split over $M_0$.
\end{corollary}
\begin{proof}
By \ref{extnsp}, extend $p\restriction M_1$ to a nonsplitting $q\in\gs(M_2)$. By \ref{uninsp} we have $p=q$ so $p$ does not $\mu$-split over $M_0$.
\end{proof}\fi

For local character, we imitate \cite[Lemma 4.11]{s3} which handled the case of $\mu^+$-saturated models ordered by $\lk$ instead of $<_u$. That proof originates from \cite[II Claim 2.11(5)]{shh}.
\begin{proposition}[Local character]\mylabel{gfl}{Proposition \thetheorem}
Let $\delta\geq\chi$ be regular, $\langle M_i:i\leq\delta\rangle\subseteq K_{\geq\mu}$ u-increasing and continuous, $p\in\gs(M_\delta)$. There is $i<\delta$ such that $p$ does not $\mu$-fork over $M_i$. 
\end{proposition}
\begin{proof}
If $\delta\geq\mu^+$, then by existence of nonsplitting (\ref{eeuprop}) and monotonicity, there is $j<\delta$ such thtat $p$ does not $\mu$-split over $M_j$. As $M_{j+1}$ is universal over $M_j$, $p$ does not $\mu$-fork over $M_{j+1}$. 

If $\chi\leq\delta\leq\mu$ and suppose the conclusion fails, then we can build 
\begin{enumerate}
\item $\langle N_i:i\leq\delta\rangle\subseteq K_{\mu}$ u-increasing and continuous;
\item $\langle N_i':i\leq\delta\rangle\subseteq K_{\mu}$ increasing and continuous;
\item $N_0=N_0'\leq M_0$ be any model in $K_\mu$;
\item For all $i<\delta$, $N_i\leq M_i$ and $N_i\leq N_i'\leq M_\delta$.
\item For all $i<\delta$, $\bigcup_{j\leq i}(|N_j'|\cap |M_{i+1}|)\subseteq|N_{i+1}|$
\item For all $j<\delta$, $p\restriction N_{j+1}'$ $\mu$-splits over $N_j$.
\end{enumerate}
We specify the successor step of $N_i'$: suppose $N_i$ has been constructed. Since $p$ $\mu$-forks over $M_i$, hence over $N_i$. Thus $(N_{i-1},N_i)$ cannot witness nonforking, so there is $N_i'\in K_\mu$ with $N_i\leq N_i'\leq M_\delta$ such that $p\restriction N_i'$ $\mu$-splits over $N_{i-1}$. After the construction, by monotonicity $p\restriction N_\delta\supseteq p\restriction N_i'$ $\mu$-splits over $N_{i-1}$ for each successor $i$, contradicting $\chi$-local character of $\mu$-nonsplitting.
\end{proof}
In Section 6, we will need the original form of \cite[Lemma 4.11]{s3}, whose proof is similar to \ref{gfl}. We write the statement here for comparison.
\begin{fact}\mylabel{gflfact}{Fact \thetheorem}
Let $\delta\geq\chi$ be regular, $\langle M_i:i\leq\delta\rangle$ be an increasing and continuous chain of $\mu^+$-saturated models, $p\in\gs(M_\delta)$. There is $i<\delta$ such that $p$ does not $\mu$-fork over $M_i$. 
\end{fact}

We now show the promised continuity of nonforking. In \cite[Lemma 4.12]{s3}, the chain must be of length $\geq\chi$. We do not have the restriction here because we have continuity of nonsplitting in \ref{assum1}.
\begin{proposition}[Continuity 2]\mylabel{gfcup}{Proposition \thetheorem}
Let $\delta<\mu^+$ be regular, $\langle M_i:i\leq\delta\rangle\subseteq K_{\geq\mu}$ u-increasing and continuous, and ${\bf K}$ is stable in $[\nr{M_1},\nr{M_\delta})$. Let $p\in\gs(M_\delta)$ satisfy $p\restriction M_i$ does not $\mu$-fork over $M_0$ for all $1\leq i<\delta$. Then $p$ also does not $\mu$-fork over $M_0$.
\end{proposition}
\begin{proof}
If $\delta\geq\chi$, by \ref{gfl} there is $i<\delta$ such that $p$ does not $\mu$-fork over $M_i$. By \ref{gft}, $p$ does not $\mu$-fork over $M_0$. 

If $\delta<\chi\leq\mu$, we have two cases: (1) $M_0\in K_\mu$: then for $1\leq i<\delta$, $p\restriction M_i$ does not $\mu$-split over $M_0$. By \ref{congeqmu}, $p$ does not $\mu$-split over $M_0$, so $p$ does not $\mu$-fork over $M_1$. By \ref{gft}, $p$ does not $\mu$-fork over $M_0$. (2) $M_0\in K_{>\mu}$: for $1\leq i<\delta$, let $N_i<_u M_0$ witness $p\restriction M_i$ does not $\mu$-fork over $M_0$. By  L\"{o}wenheim-Skolem axiom, there is $N\in K_\mu$ (here we need $\delta\leq\mu$) such that $N_i<_uN\leq M_0$ for all $i$. Apply case (1) with $N$ replacing $M_0$.
\end{proof}
Existence is more tricky because nonforking requires the base to be universal over the witness of nonsplitting. The second part of the proof is based on \cite[Lemma 4.9]{s3}.
\begin{proposition}[Existence]\mylabel{gfexi}{Proposition \thetheorem}
Let $M$ be a $(\geq\mu,\geq\chi)$-limit model, $p\in\gs(M)$. Then $p$ does not $\mu$-fork over $M$. Alternatively $M$ can be a $\mu^+$-saturated model.
\end{proposition}
\begin{proof}
The first part is immediate from \ref{gfl}. For the second part, apply existence of nonsplitting \ref{eeuprop} to obtain $N\in K_\mu$, $N\leq M$ such that $p$ does not $\mu$-split over $N$. By model-homogeneity, $M$ is universal over $N$, hence $p$ does not $\mu$-fork over $M$.
\end{proof}
\begin{corollary}\mylabel{gfcor}{Corollary \thetheorem}
There exists a good $\mu$-frame over the $\mu$-skeleton of $(\mu,\geq\chi)$-limit models ordered by $\leq_u$, except for symmetry and local character $\chi$ in place of $\al$. 
\end{corollary}
\begin{proof}
Define nonforking as in \ref{nfdef}(2). Invariance and monotonicity are immediate. Existence is by \ref{gfexi}, $\chi$-local character is by \ref{gfl}, extension is by \ref{gfe}, uniqueness is by \ref{gfu}, continuity is by \ref{gfc}.
\end{proof}
\begin{remark}\mylabel{gfcorrmk}{Remark \thetheorem}
\begin{enumerate}
\item We do not expect $\aleph_0$-local character because there are strictly stable AECs. For the same reason we restrict models to be $(\mu,\geq\chi)$-limit to guarantee existence property. 
\item Let $\lambda\geq\mu$. Our frame extends to $([\mu,\lambda],\geq\chi)$-limit models if we assume stability in $[\mu,\lambda]$. However \cite{s3} has already developed $\mu$-nonforking for $\mu^+$-saturated models \emph{ordered by $\leq$}, and we will see in \ref{unicor}(2) that under extra stability assumptions, $(>\mu,\geq\chi)$-limit models are automatically $\mu^+$-saturated, so the interesting part is $K_\mu$ here.
\item We will see in \ref{symcor2}(2) that symmetry also holds if we have enough stability.
\end{enumerate}
\end{remark}
Since we have built an approximation of a good frame in \ref{gfcor}, one might ask if it is canonical. We first observe the following fact (\ref{assum1} is not needed):  
\begin{fact}\cite[Theorem 14.1]{snote}\mylabel{scanon}{Fact \thetheorem}
Let $\lambda\geq\ls$. Suppose ${\bf K}$ is $\lambda$-superstable and there is an independence relation over the limit models (ordered by $\leq$) in $K_\lambda$, satisfying invariance, monotonicity, universal local character, uniqueness and extension. Let $M\leq N$ be limit models in $K_\lambda$ and $p\in\gs(N)$. Then $p$ is independent over $M$ iff $p$ does not $\lambda$-fork over $M$.  
\end{fact}
Its proof has the advantage that it does not require the independence relation to be for longer types as in \cite[Corollary 5.19]{BGKV}. However, it still uses the following lemma from \cite[Lemma 4.2]{BGKV}:
\begin{lemma}\mylabel{bgkvlem}{Lemma \thetheorem}
Suppose there is an independence relation over models in $K_\mu$ ordered by $\leq$. If it satisfies invariance, monotonicity and uniqueness, then the relation is extended by $\mu$-nonsplitting.
\end{lemma}
\begin{proof}
Suppose $M\leq N$ in $K_\mu$, $p\in\gs(N)$ is independent over $M$. For any $N_1,N_2\in K_\mu$ with $M\leq N_1,N_2\leq N$, and any $f:N_1\cong_MN_2$. We need to show that $f(p)\restriction N_2=p\restriction N_2$. By monotonicity, $p\restriction N_1$ and $p\restriction N_2$ do not depend on $M$. By invariance, $f(p)\restriction N_2$ is independent over $M$. By uniqueness and the fact that $f$ fixes $M$, we have $f(p)\restriction N_2=p\restriction N_2$.
\end{proof}
In the above proof, it utilizes the assumption that the independence relation is for models ordered by $\leq$, so it makes sense to talk about $p\restriction N_i$ is independent over $M$ for $i=1,2$. To generalize \ref{scanon} to our frame in \ref{gfcor}, one way is to assume the independence relation to be for models ordered by $\leq$, and with universal local character $\chi$. But since we defined our frame to be for models ordered by $\leq_u$, we want to keep the weaker assumption that the arbitrary independence relation is also for models ordered by $\leq_u$. Thus we cannot directly invoke \ref{bgkvlem}, where the $N_i$'s are not necessarily universal over $M$. To circumvent this, we adapt the lemma by allowing more room:
\begin{lemma}\mylabel{bgkvlem2}{Lemma \thetheorem}
Let $M<_uN<_uN'$ all in $K_\mu$, $p\in\gs(N')$. If $p\restriction N$ $\mu$-splits over $M$, then $p$ also $\mu$-splits over $M$ with witnesses universal over $M$. Namely, there are $N_1',N_2'\leq N'$ such that $N_1'>_uM$, $N_2'>_uM$ and there is $f':N_1'\cong_MN_2'$ with $f(p)\restriction N_2'\neq p\restriction N_2'$.
\end{lemma}
\begin{proof}
By assumption, there are $N_1,N_2\in K_\mu$ such that $M\leq N_1,N_2\leq N$ and there is $f:N_1\cong_MN_2$ such that $f(p\restriction N)\restriction N_2\neq p\restriction N_2$. Extend $f$ to an isomorphism $\tilde{f}$ of codomain $N$, and let $N^*\geq N_1$ be the domain of $\tilde{f}$. Since $N>_uM$, by invariance $N^*>_uM$. On the other hand, $N'>_uN$, then $N'>_uN_1$ and there is $g:N^*\xrightarrow[N_1]{}N'$. Let the image of $g$ be $N^{**}$

In the diagram below, we use dashed arrows to indicate isomorphisms. Solid arrows indicate $\leq$.
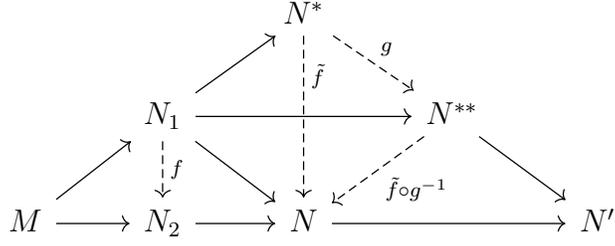
\begin{center}
\begin{tikzcd}
                       &                                                             & N^* \arrow[dd, near start,"\tilde{f}", dashed] \arrow[rd, "g", dashed] &                                                               &    \\
                       & N_1 \arrow[ru] \arrow[d, "f", dashed] \arrow[rd] \arrow[rr] &                                                             & N^{**} \arrow[rd] \arrow[ld, "\tilde{f}\circ g^{-1}", dashed] &    \\
M \arrow[ru] \arrow[r] & N_2 \arrow[r]                                               & N \arrow[rr]                                                &                                                               & N'
\end{tikzcd}
\end{center}
Since $\tilde{f}\circ g^{-1}:N^{**}\cong_{M}N$ and $M<_u N^{**},N\leq N'$, we consider $\tilde{f}\circ g^{-1}(p)\restriction N$ and $p\restriction N$. 
\begin{align*}
\tilde{f}\circ g^{-1}(p)\restriction N&\geq[\tilde{f}\circ g^{-1}(p)]\restriction N_2\\
&=\tilde{f}([g^{-1}(p)]\restriction N_1)\restriction N_2\quad\text{as $\tilde{f}^{-1}[N_2]=N_1$}\\
&=\tilde{f}(p\restriction N_1)\restriction N_2\quad\text{as $g$ fixes $N_1$}\\
&=f(p\restriction N_1)\restriction N_2\quad\text{as $\tilde{f}$ extends $f$}\\
&=f(p\restriction N)\restriction N_2\quad\text{as $f^{-1}[N_2]=N_1\leq N$}\\
p\restriction N&\geq p\restriction N_2
\end{align*}
Since $f(p\restriction N)\restriction N_2\neq p\restriction N_2$, $\tilde{f}\circ g^{-1}(p)\restriction N\neq p\restriction N$ and we can take $N_1'\defeq N^{**}$, $N_2'\defeq N$, $f'\defeq\tilde{f}\circ g^{-1}$ in the statement of the lemma. 
\end{proof}
Now we can prove a canonicity result for our frame. In order to apply \ref{bgkvlem2}, we will need to enlarge $N$ to a universal extension in order to have more room. This procedure is absent in the original forward direction of \ref{scanon} but is similar to the backward direction (to get $q$ below). 
\begin{proposition}\mylabel{canon}{Proposition \thetheorem}
Suppose there is an independence relation over the $(\mu,\geq\chi)$-limit models ordered by $\leq_u$ satisfying invariance, monotonicity, local character $\chi$, uniqueness and extension. Let $M<_u N$ be $(\mu,\geq\chi)$-limit models and $p\in\gs(N)$. Then $p$ is independent over $M$ iff $p$ does not $\mu$-fork over $M$.  
\end{proposition}
\begin{proof}
Suppose $p$ is independent over $M$. By assumption $M$ is a $(\mu,\delta)$-limit for some regular $\delta\in[\chi,\mu^+)$. Resolve $M=\bigcup_{i<\delta}M_i$ such that all $M_i$ are also $(\mu,\delta)$-limit. By local character, $p\restriction M$ is independent over $M_i$ for some $i<\delta$. Since the independence relation satisfies uniqueness and extension, by the proof of \ref{gft} it also satisfies transitivity. Therefore $p$ is independent over $M_i$. Let $N'>_uN$. By extension, there is $p'\in\gs(N')$ independent over $M_i$ and $p'\supseteq p$. Now suppose $p$ $\mu$-splits over $M_i$, by \ref{bgkvlem2} $p'$ $\mu$-splits over $M_i$ with universal witnesses, contradicting \ref{bgkvlem} (where $\leq$ is replaced by $<_u$ where). As a result, $p$ does not $\mu$-split over $M_i$. Since $M_i<_u M$, $p$ does not $\mu$-fork over $M$. 

Conversely suppose $p$ does not $\mu$-fork over $M$. By local character and monotonicity, $p\restriction M$ is independent over $M$. By extension, obtain $q\in\gs(N)$ independent over $M$ and $q\supseteq p$. From the forward direction, $q$ does not $\mu$-fork over $M$. By \ref{gfu}, $p=q$ so invariance gives $q$ independent over $M$.  
\end{proof}
To conclude this section, we show that the existence of a frame similar to \ref{gfcor} is sufficient to obtain local character of nonsplitting. Continuity of $\mu$-nonsplitting and $\mu$-tameness in \ref{assum1} are not needed.
\begin{proposition}\mylabel{getloc}{Proposition \thetheorem}
Let $\delta<\mu^+$ be regular. Suppose there is an independence relation over the $(\mu,\geq\delta)$-limit models ordered by $\leq_u$ satisfying invariance, monotonicity, local character $\delta$, uniqueness and extension. Then ${\bf K}$ has $\delta$-local character of $\mu$-nonsplitting.
\end{proposition}
\begin{proof}
Let $\langle M_i:i\leq\delta\rangle$ be u-increasing and continuous, $p\in\gs(M_\delta)$. There is $i<\delta$ such that $p$ is independent over $M_i$. By the forward direction of \ref{canon} (local character of nonsplitting is not used), $p$ does not $\mu$-split over $M_i$.
\end{proof}
\section{Local symmetry}\label{ufsec5}
Tower analysis was used in \cite[Theorem 3]{van16a} to connect a notion of $\mu$-symmetry and reduced towers. Combining with \cite{GVV}, superstability and $\mu$-symmetry imply the uniqueness of limit models. \cite[Lemma 4.6]{vv} observed that a weaker form of $\mu$-symmetry is sufficient to deduce one direction of \cite[Theorem 3]{van16a}, and enough superstability implies the weaker form of $\mu$-symmetry. Therefore enough superstability already implies the uniqueness of limit models \cite[Corollary 1.4]{vv}. Meanwhile, \cite{bvulm} localized the notion of $\mu$-symmetry and deduced the uniqueness of limit models of length $\geq\chi$. We will imitate the above argument and replace the hypothesis of local symmetry by sufficient stability. As a corollary we will obtain symmetry property of nonforking. The uniqueness of limit models will be discussed in the next section.

The following is based on \cite[Definition 10]{bvulm}. They restricted $M_0$ to be exactly $(\mu,\delta)$-limit over $N$ but they should mean $(\mu,\geq\delta)$ for the proofs to go through. We will use $\delta\defeq\chi$ in this paper.
\begin{definition}\mylabel{musymdef}{Definition \thetheorem}
Let $\delta<\mu^+$ be a limit ordinal. ${\bf K}$ has \emph{$(\mu,\delta)$-symmetry for $\mu$-nonsplitting} if for any $M,M_0,N\in K_\mu$, elements $a,b$ with
\begin{enumerate}
\item $a\in M-M_0$;
\item $M_0<_uM$ and $M_0$ is $(\mu,\geq\delta)$-limit over $N$;
\item $\gtp(a/M_0)$ does not $\mu$-split over $N$;
\item $\gtp(b/M)$ does not $\mu$-split over $M_0$,
\end{enumerate}
then there is $M^b\in K_\mu$ universal over $M_0$ and containing $b$ such that $\gtp(a/M^b)$ does not $\mu$-split over $N$. We will abbreviate $(\mu,\delta)$-symmetry for $\mu$-nonsplitting as \emph{$(\mu,\delta)$-symmetry}.
\end{definition}
Now we localize the hierarchy of symmetry properties in \cite[Definition 4.3]{vv}. The first two items will be important in our improvement of \cite{bvulm}.
\begin{definition}\mylabel{symhier}{Definition \thetheorem}
Let $\delta<\mu^+$ be a limit ordinal. In the following items, we always let $a\in M-M_0$, $M_0<_uM$, $M_0$ be $(\mu,\geq\delta)$-limit over $N$ and $b$ be an element. In the conclusion, $M^b\in K_\mu$ universal over $M_0$ and containing $b$ is guaranteed to exist.
\begin{enumerate}
\item ${\bf K}$ has \emph{uniform $(\mu,\delta)$-symmetry}: If $\gtp(b/M)$ does not $\mu$-split over $M_0$, $\gtp(a/M_0)$ does not $\mu$-fork over $(N,M_0)$, then $\gtp(a/M^b)$ does not $\mu$-fork over $(N,M_0)$.
\item ${\bf K}$ has \emph{weak uniform $(\mu,\delta)$-symmetry}: If $\gtp(b/M)$ does not $\mu$-fork over $M_0$, $\gtp(a/M_0)$ does not $\mu$-fork over $(N,M_0)$, then $\gtp(a/M^b)$ does not $\mu$-fork over $(N,M_0)$.
\item ${\bf K}$ has \emph{nonuniform $(\mu,\delta)$-symmetry}: If $\gtp(b/M)$ does not $\mu$-split over $M_0$, $\gtp(a/M_0)$ does not $\mu$-fork over $M_0$, then $\gtp(a/M^b)$ does not $\mu$-fork over $M_0$.
\item ${\bf K}$ has \emph{weak nonuniform $(\mu,\delta)$-symmetry}: If $\gtp(b/M)$ does not $\mu$-fork over $M_0$, $\gtp(a/M_0)$ does not $\mu$-fork over $M_0$, then $\gtp(a/M^b)$ does not $\mu$-fork over $M_0$.
\end{enumerate}
\end{definition}
The following results generalize \cite[Section 4]{vv} which assumes superstability and works with full symmetry properties. 
\begin{proposition}\mylabel{unisymeq}{Proposition \thetheorem}
Let $\delta<\mu^+$ be a limit ordinal. $(\mu,\delta)$-symmetry is equivalent to uniform $(\mu,\delta)$-symmetry. Both imply nonuniform $(\mu,\delta)$-symmetry and weak uniform $(\mu,\delta)$-symmetry. Nonuniform $(\mu,\delta)$-symmetry implies weak nonuniform $(\mu,\delta)$-symmetry.
\end{proposition}
\begin{proof}
In the definition of the symmetry properties, we always have $N<_uM_0$, so the following are equivalent:
\begin{itemize}\item $\gtp(a/M_0)$ does not $\mu$-fork over $(N,M_0)$;
\item $\gtp(a/M_0)$ does not $\mu$-split over $N$.
\end{itemize}
Similarly, the following are equivalent:
\begin{itemize}\item $\gtp(a/M^b)$ does not $\mu$-fork over $(N,M_0)$;
\item $\gtp(a/M^b)$ does not $\mu$-split over $N$.
\end{itemize}
Therefore, $(\mu,\delta)$-symmetry is equivalent to uniform $(\mu,\delta)$-symmetry.

Uniform $(\mu,\delta)$-symmetry implies weak uniform $(\mu,\delta)$-symmetry because nonforking over $M_0$ is a stronger assumption than nonsplitting over $M_0$. Uniform $(\mu,\delta)$-symmetry implies nonuniform $(\mu,\delta)$-symmetry because the latter does not require the witness to nonforking be the same, so its conclusion is weaker. Nonuniform $(\mu,\delta)$-symmetry implies weak nonuniform $(\mu,\delta)$-symmetry because nonforking over $M_0$ is a stronger assumption than nonsplitting over $M_0$.
\end{proof}
The following result modifies the proof of \cite{bvulm} which involves a lot of tower analysis. We will only mention the modifications and refer the readers to the original proof.
\begin{proposition}\mylabel{uniwkeq}{Proposition \thetheorem}
Let $\delta<\mu^+$ be a limit ordinal. If $\delta\geq\chi$, then weak uniform $(\mu,\delta)$-symmetry implies uniform $(\mu,\delta)$-symmetry.
\end{proposition}
\begin{proof}[Proof sketch]
\cite[Theorems 18, Proposition 19]{bvulm} establish that $(\mu,\delta)$-symmetry is equivalent to continuity of reduced towers at $\geq\delta$. We will show that the backward direction only requires weak uniform $(\mu,\delta)$-symmetry. Then using the equivalence twice we deduce that weak uniform $(\mu,\delta)$-symmetry implies $(\mu,\delta)$-symmetry. By the previous proposition, it is equivalent to uniform $(\mu,\delta)$-symmetry.

There are three places in \cite[Theorems 18]{bvulm} which use $(\mu,\delta)$-symmetry. In the first two paragraphs of page 11:
\begin{enumerate}
\item By $\chi$-local character, there is a successor $i^*<\delta$ such that $\gtp(b/M_\delta^\delta)$ does not $\mu$-split over $M_{i^*}^{i^*}$. 
\item For any $j<\delta$, $M_\delta^\delta$ is universal over $M_{j}^{j}$.
\item For any $j<\delta$, $\gtp(a_{j}/M_{j}^{j})$ does not $\mu$-split over $N_{j}$. 
\item For any successor $j<\delta$, $M_j^j$ is $(\mu,\geq\delta)$-limit over $M_{j-1}^{j-1}$ and over $N_j$. 
\end{enumerate}
Let $j^*\defeq i^*+1$ which is still a successor ordinal less than $\delta$. Combining (1) and (4), we have $\gtp(b/M_\delta^\delta)$ does not $\mu$-fork over $M_{j^*}^{j^*}$. Combining (3) and (4), $\gtp(a_{j^*}/M_{j^*}^{j^*})$ does not $\mu$-fork over $M_{j^*}^{j^*})$. Moreover, (2) gives $M_\delta^\delta$ is universal over $M_{j^*}^{j^*}$. Together with (4) and weak uniform $(\mu,\delta)$-symmetry, we can find $M^b$ $(\mu,\geq\delta)$-limit over $M_{j^*}^{j^*}$ and containing $b$ such that $\gtp(a/M^b)$ does not $\mu$-fork over $(N_{j^*},M_{j^*}^{j^*})$. In other words, $\gtp(a/M^b)$ does not $\mu$-split over $N_{j^*}$ and so the original argument goes through with $i^*$ replaced by $j^*$.

In ``Case 2'' on page 12:
\renewcommand{\theenumi}{\alph{enumi}}
\begin{enumerate}
\item $\gtp(b/\bigcup_{l<\alpha}M_l^l)$ does not $\mu$-split over $M_{i^*}^{i^*}$.
\item $i^*+2\leq k<\alpha$ and $\gtp(a_k/M_k^{k+1})$ does not $\mu$-split over $N_k$.
\item  $M_k^{k+1}$ is universal over $M_{i^*}^{i^*}$.
\item $\bigcup_{l<\alpha}M_l^l$ is universal over $M_k^{k+1}$. $M_k^{k+1}$ is $(\mu,\geq\delta)$-limit over $N_k$.
\end{enumerate}
Combining (a) and (c), $\gtp(b/\bigcup_{l<\alpha}M_l^l)$ does not $\mu$-fork over $M_k^{k+1}$. (b) gives $\gtp(a_k/M_k^{k+1})$ does not $\mu$-fork over $(N_k,M_k^{k+1})$. Together with (d) and weak uniform $(\mu,\delta)$-symmetry, we can find $M^b_k$ $(\mu,\geq\delta)$-limit over $M_{k}^{k+1}$ and containing $b$ such that $\gtp(a_k/M^b_k)$ does not $\mu$-fork over $(N_{k},M_{k}^{k+1})$ so the proof goes through (we do not change index this time).

Before ``Case 1'' on page 11, they refer the successor case to the original proof of \cite[Theorem 3]{van16a} which also uses $(\mu,\delta)$-symmetry. But the idea from the previous case applies equally.
\end{proof}
In \cite[Corollary 2.18]{s22}, it was shown that under superstability, weak nonuniform $\mu$-symmetry implies weak uniform $\mu$-symmetry. We generalize this as:
\begin{proposition}\mylabel{wknonwkeq}{Proposition \thetheorem}
Let $\delta<\mu^+$ be a limit ordinal. Weak nonuniform $(\mu,\delta)$-symmetry implies weak uniform $(\mu,\delta)$-symmetry. 
\end{proposition}
\begin{proof}
Using the notation in \ref{symhier}, we assume $\gtp(b/M)$ does not $\mu$-fork over $M_0$ and $\gtp(a/M_0)$ does not $\mu$-fork over $(N,M_0)$. By  weak nonuniform $(\mu,\delta)$-symmetry, we can find $M^b$ such that $\gtp(a/M^b)$ does not $\mu$-fork over $M_0$. Since $\gtp(a/M_0)$ does not $\mu$-fork over $(N,M_0)$, by extension of nonsplitting (\ref{eeuprop}), there is $a'$ such that $\gtp(a/M_0)=\gtp(a'/M_0)$ and $\gtp(a'/M^b)$ does not $\mu$-split over $N$. Now both $\gtp(a/M^b)$ and $\gtp(a'/M^b)$ do not $\mu$-fork over $M_0$ and they agree on the restriction of $M_0$. By uniqueness of nonforking (\ref{gfu}), $\gtp(a/M^b)=\gtp(a'/M^b)$ and hence $\gtp(a/M^b)$ does not $\mu$-split over $N$. In other words, it does not $\mu$-fork over $(N,M_0)$ as desired.
\end{proof}
\begin{corollary}\mylabel{symcor}{Corollary \thetheorem}
The following are equivalent:
\begin{enumerate}\setcounter{enumi}{-1}
\item $(\mu,\chi)$-symmetry for $\mu$-nonsplitting;
\item Uniform $(\mu,\chi)$-symmetry;
\item Weak uniform $(\mu,\chi)$-symmetry;
\item Nonuniform $(\mu,\chi)$-symmetry;
\item Weak nonuniform $(\mu,\chi)$-symmetry.
\end{enumerate}
\end{corollary}
\begin{proof}
By \ref{unisymeq}, (0) and (1) are equivalent, (1) implies (2) and (3) while (3) implies (4). By \ref{uniwkeq} (this is where we need $\chi$ instead of a general $\delta$), (2) implies (1). By \ref{wknonwkeq}, (4) implies (2). 
\end{proof}

The following adapts \cite[Lemma 5.6]{vv} and fills in some gaps. In particular we need $\mu$-tameness (in \ref{assum1}) and stability in $\nr{N_\alpha}$ for the proof to go through. It is not clear how to remove $\mu$-tameness which they do not assume. 
\begin{lemma}\mylabel{symlem}{Lemma \thetheorem}
Let $M_0\in K_\mu$, $N_\alpha\in K_{\geq\mu}$ with $M_0\leq N_\alpha$, $b,b_\beta\in |N_\alpha|$, $a_\alpha$ be an element. If ${\bf K}$ is stable in $\nr{N_\alpha}$, $\gtp(a_\alpha/N_\alpha)$ does not $\mu$-fork over $M_0$ and $\gtp(b/M_0)=\gtp(b_\beta/M_0)$, then $\gtp(a_\alpha b/M_0)=\gtp(a_\alpha b_\beta/M_0)$.
\end{lemma}
\begin{proof}
Let $M^*<_u M_0$ witness that $\gtp(a_\alpha/N_\alpha)$ does not $\mu$-fork over $(M^*,M_0)$. By extension (\ref{extnsp}) and weak uniqueness of nonsplitting (\ref{eeuprop}(2)), we can extend $N_\alpha$ to $N^*>_u N_\alpha$ such that $\gtp(a_\alpha/N^*)$ does not $\mu$-split over $M^*$. As $\gtp(b/M_0)=\gtp(b_\beta/M_0)$ and $N^*>_uN_\alpha$, there is $f:N_\alpha\xrightarrow[M_0]{}{N^*}$ such that $f(b)=b_\beta$. As $\gtp(a_\alpha/N^*)$ does not $\mu$-split over $M^*$, by \ref{tamespl} $\gtp(f(a_\alpha)/f(N_\alpha))=\gtp(a_\alpha/f(N_\alpha))$. Hence there is $g\in\oop{Aut}_{f(N_\alpha)}(\mn)$ such that $g(f(a_\alpha))=a_\alpha$. Then
$$\gtp(a_\alpha b/M_0)=\gtp(f(a_\alpha)f(b)/M_0)=\gtp(g(f(a_\alpha))f(b)/M_0)=\gtp(a_\alpha b_\beta/M_0).$$
\end{proof}
\begin{remark}\mylabel{symlemrmk}{Remark \thetheorem}
By swapping the dummy variables, we have the following formulation: Let $M_0\in K_\mu$, $N_\beta'\in K_{\geq\mu}$ with $M_0\leq N_\beta'$, $a,a_\alpha\in |N_\beta'|$, $b_\beta$ be an element. If ${\bf K}$ is stable in $\nr{N_\beta'}$, $\gtp(b_\beta/N_\beta')$ does not $\mu$-fork over $M_0$ and $\gtp(a/M_0)=\gtp(a_\alpha/M_0)$, then $\gtp(ab_\beta/M_0)=\gtp(a_\alpha b_\beta/M_0)$.
\end{remark}
The following adapts \cite[Lemma 5.7]{vv} which assumes superstability in $[\mu,\lambda)$. When we write the \emph{$\mu$-order property}, we mean tuples that witness order property have length $\mu$.
\begin{proposition}\mylabel{sympf}{Proposition \thetheorem}
Let $\lambda\geq\mu$ be a cardinal. If ${\bf K}$ is stable in $[\mu,\lambda)$ and fails $(\mu,\chi)$-symmetry, then it has the $\mu$-order property of length $\lambda$.
\end{proposition}
\begin{proof}
By \ref{symcor}(2)$\Rightarrow$(0), ${\bf K}$ fails weak uniform $(\mu,\chi)$-symmetry. So there are $N,M_0,M\in K_\mu$ and elements $a,b$ such that 
\begin{itemize}
\item$a\in M-M_0$, $M_0<_uM$ and $M_0$ is $(\mu,\geq\chi)$-limit over $N$;
\item $\gtp(b/M)$ does not $\mu$-fork over $M_0$;
\item $\gtp(a/M_0)$ does not $\mu$-fork over $(N,M_0)$;
\item There is no $M^b\in K_\mu$ universal over $M_0$ containing $b$ such that $\gtp(a/M^b)$ does not $\mu$-fork over $(N,M_0)$.
\end{itemize}
Build $\langle a_\alpha,b_\alpha,N_\alpha,N_\alpha':\alpha<\lambda\rangle$ such that:
\begin{enumerate}
\item $N_\alpha,N_\alpha'\in K_{\mu+|\alpha|}$;
\item $b\in |N_0|$ and $N_0$ is universal over $M$;
\item $N_\alpha<_u N_\alpha'<_uN_{\alpha+1}$;
\item $a_\alpha\in|N_\alpha'|$ and $\gtp(a_\alpha/M_0)=\gtp(a/M_0)$;
\item $b_\alpha\in|N_{\alpha+1}|$ and $\gtp(b_\alpha/M)=\gtp(b/M)$;
\item $\gtp(a_\alpha/N_\alpha)$ does not $\mu$-fork over $(N,M_0)$;
\item $\gtp(b_\alpha/N_\alpha')$ does not $\mu$-fork over $M_0$.
\end{enumerate}
$N_0$ is specified in (2). We specify the successor step: suppose $N_\alpha$ has been constructed , by \ref{extnsp} there is $a_\alpha$ such that $\gtp(a_\alpha/N_\alpha)$ extends $\gtp(a/M_0)$ and does not $\mu$-fork over $(N,M_0)$. Build any $N_\alpha'$ universal over $N_\alpha$ containing $a_\alpha$. By \ref{gfe} again, there is $b_\alpha$ such that $\gtp(b_\alpha/N_\alpha')$ extends $\gtp(b/M)$ and does not $\mu$-fork over $M_0$. Build $N_{\alpha+1}$ universal over $N_\alpha'$ containing $b_\alpha$. Notice that stability is used to guarantee the existence of $N_\alpha,N_\alpha'$ and the extension of types.

After the construction, we have the following properties for $\alpha,\beta<\lambda$:
\renewcommand{\theenumi}{\alph{enumi}}
\begin{enumerate}
\item $\gtp(a_\alpha b/M_0)\neq\gtp(ab/M_0)$;
\item $\gtp(ab_\beta/M_0)=\gtp(ab/M_0)$;
\item If $\beta<\alpha$, $\gtp(ab/M_0)\neq\gtp(a_\alpha b_\beta/M_0)$;
\item If $\beta\geq\alpha$, $\gtp(ab/M_0)=\gtp(a_\alpha b_\beta/M_0)$.
\end{enumerate}
Suppose (a) is false. By invariance and the choice of $a,b,M_0,N$ there is no $M'\in K_\mu$ universal over $M_0$ containing $b$ such that $\gtp(a_\alpha/M')$ does not $\mu$-fork over $(N,M_0)$. This contradicts $M'\defeq N_\alpha$ and item (6) in the construction. (b) is true because of item (5) of the construction and $a\in|M|$. For (c), items (5), (6) and \ref{symlem} (with the exact same notations) imply $\gtp(a_\alpha b_\beta/M_0)=\gtp(a_\alpha b/M_0)$ which is not equal to $\gtp(ab/M_0)$ by (a). For (d), items (4), (7) and \ref{symlemrmk} imply $\gtp(a_\alpha b_\beta/M_0)=\gtp(ab_\beta/M_0)$ which is equal to $\gtp(ab/M_0)$ by (b).

To finish the proof, let $d$ enumerate $M_0$, and for $\alpha<\lambda$, $c_\alpha\defeq a_\alpha b_\alpha d$. By (c) and (d) above, $\langle c_\alpha:\alpha<\lambda\rangle$ witnesses the $\mu$-order property of length $\lambda$.
\end{proof}
\begin{remark}
When proving (d), we used \ref{symlemrmk} which requires $\gtp(b_\beta/N_\beta')$ nonforking over $M_0$, and this is from extending $\gtp(b/M)$ nonforking over $M_0$. This called for the failure of weak uniform $(\mu,\chi)$-symmetry instead of just $(\mu,\chi)$-symmetry. (In the original proof, they claimed the same for (c) in place of (d), which should be a typo.)
\end{remark}
\begin{ques}
Is it possible to weaken the stability assumption in \ref{sympf}?
\end{ques}
\begin{fact}\mylabel{symfact}{Fact \thetheorem}
For any infinite cardinal $\lambda$, $h(\lambda)\defeq\beth_{(2^{\lambda})^+}$. When we write the $\mu$-stable, we mean stability of tuples of length $\mu$.
\begin{enumerate}
\item \cite[Claim 4.6]{sh394} If ${\bf K}$ does not have the $\mu$-order property, then there is $\lambda<h(\mu)$ such that ${\bf K}$ does not have the $\mu$-order property of length $\lambda$.
\item \cite[Fact 5.13]{BGKV} If ${\bf K}$ is $\mu$-stable (in some cardinal $\geq\mu$), then it does not have the $\mu$-order property.
\item If ${\bf K}$ is stable in some $\lambda=\lambda^\mu$, then ${\bf K}$ is $\mu$-stable in $\lambda$.
\item \cite[Corollary 6.4]{GV06b} If ${\bf K}$ is stable and tame in $\mu$ (these are in \ref{assum1}), then it is stable in all $\lambda=\lambda^\mu$. In particular it is stable in $2^\mu$.
\item For some $\lambda<h(\mu)$, ${\bf K}$ does not have the $\mu$-order property of length $\lambda$.
\end{enumerate}
\end{fact}
\begin{proof}
For (1) and (2), see also \cite[Proposition 3.4]{leung1} for a proof sketch. (3) is an immediate corollary of \cite[Theorem 3.1]{bon3.1}, see \cite[Theorem 2.1]{leung1} for a proof. %\cite[Theorem 2.9]{leung1} proves (4) with more assumptions.
We show (5): by (4) ${\bf K}$ is stable in $2^\mu$. By (3) it is $\mu$-stable in $2^\mu$. Combining with (2) and (1) gives the conclusion.
\end{proof}
\begin{corollary}\mylabel{symcor2}{Corollary \thetheorem}
There is $\lambda<h(\mu)$ such that if ${\bf K}$ is stable in $[\mu,\lambda)$, then 
\begin{enumerate}
\item ${\bf K}$ has $(\mu,\chi)$-symmetry;
\item the frame in \ref{gfcor} satisfies symmetry.
\end{enumerate}
\end{corollary}
\begin{proof}
\begin{enumerate}
\item By \ref{symfact}(5), there is $\lambda<h(\mu)$ such that ${\bf K}$ does not have the $\mu$-order property of length $\lambda$. By the contrapositive of \ref{sympf}, ${\bf K}$ has $(\mu,\chi)$-symmetry.
\item By (1) and \ref{unisymeq}, ${\bf K}$ has weak nonuniform $(\mu,\chi)$-symmetry. Compared to symmetry in a good frame, weak nonuniform $(\mu,\chi)$-symmetry has the extra assumption that $\gtp(a/M_0)$ does not $\mu$-fork over $M_0$, but this is always true by \ref{gfexi}.
\end{enumerate}
\end{proof}
\begin{remark}
From the proof of \ref{symcor2}(2), we see that if the frame in \ref{gfcor} (which is defined for $(\mu,\geq\chi)$-limits) has symmetry, then weak nonuniform $(\mu,\chi)$-symmetry, and hence all the other ones in \ref{symcor} hold.
\end{remark}
\section{Symmetry and saturated models}\label{ufsec6}
As mentioned in the previous section, \cite[Corollary 1.4]{vv} deduced symmetry from superstability and obtained the uniqueness of limit models. It is natural to localize such argument, which was partially done in 
\begin{fact}\cite[Theorem 20]{bvulm}\mylabel{ulmfact}{Fact \thetheorem}
Assume ${\bf K}$ has $(\mu,\chi)$-symmetry (together with \ref{assum1}). Then it has the uniqueness of $(\mu,\geq\chi)$-limit models: let $M_0,M_1,M_2\in K_\mu$. If both $M_1$ and $M_2$ are $(\mu,\geq\chi)$-limit over $M_0$, then $M_1\cong_{M_0}M_2$.
\end{fact}
In the original proof of the above fact, they did not assume tameness. However, we will need tameness when we remove the symmetry assumption (see also the discussion before \ref{symlem}). 
\begin{corollary}\mylabel{unicor}{Corollary \thetheorem}
There is $\lambda<h(\mu)$ such that if ${\bf K}$ is stable in $[\mu,\lambda)$,  then
\begin{enumerate}
\item ${\bf K}$ has the uniqueness of $(\mu,\geq\chi)$-limit models.
\item if also $\mu>\ls$, any $(\mu,\geq\chi)$-limit model is saturated.
\end{enumerate}
\end{corollary}
\begin{proof}
\begin{enumerate}
\item By \ref{symcor2}(1), ${\bf K}$ has $(\mu,\chi)$-symmetry. Apply \ref{ulmfact}.
\item Suppose $\mu$ is regular. Since $\chi\leq\mu$, any $(\mu,\geq\chi)$-limit is isomorphic to a $(\mu,\mu)$-limit, which is saturated. %ch
 Suppose $\mu$ is singular. Let $M$ be a $(\mu,\geq\chi)$-limit model. We show that it is $\delta$-saturated for any regular $\delta<\mu$. Since $\delta+\chi$ is a regular cardinal in $[\chi,\mu^+)$, $M$ is also $(\mu, \delta+\chi)$-limit, which implies it is $(\delta+\chi)$-saturated.
\end{enumerate}
\end{proof}
Before stating a remark, we quote a fact in order to compare Vasey's results with ours (but we will not use that fact in our paper). Continuity of $\mu$-nonsplitting in \ref{assum1} is not needed.
\begin{fact}\cite[Theorems 5.15]{bv2}\mylabel{bvfact}{Fact \thetheorem}
Let $\chi_0\geq2^\mu$ be such that ${\bf K}$ does not have the $\mu$-order property of length $\chi_0^+$, define $\chi_1\defeq(2^{2^{\chi_0}})^{+3}$, and let $\xi\geq\chi_1$. If ${\bf K}$ is stable in unboundedly many cardinals $<\xi$, then any increasing chain of $\xi$-saturated models of length $\geq\chi$ is $\xi$-saturated.
\end{fact}
\begin{remark}
We assumed enough stability to get a \emph{local} result: the same $\mu$ was considered throughout. In contrast, \cite[Theorems 6.3, 11.7]{s19} are \emph{eventual}: \ref{bvfact} was heavily used. Some of the hypotheses there require unboundedly many ($H_1$-closed) stability cardinals.
\end{remark}
Now we turn to an AEC version of Harnik's Theorem. \cite[Lemma 11.9]{s19} improved \cite[Theorem 1]{van16b} and showed that:
\begin{fact}\mylabel{s19result}{Fact \thetheorem}
Let ${\bf K}$ be $\mu$-tame with a monster model. Let $\xi\geq\mu^+$. Suppose
\begin{enumerate}
\item ${\bf K}$ is stable in $\mu$ and $\xi$;
\item $\langle M_i:i<\delta\rangle$ is an increasing chain of $\xi$-saturated models;
\item $\cf(\delta)\geq\chi$;
\item $(\xi,\delta)$-limit models are saturated,
\end{enumerate}
then $\bigcup_{i<\delta}M_i$ is $\xi$-saturated.
\end{fact}
We remove the assumption of (4) by assuming more stability and continuity of nonsplitting. Our proof is based on \cite[Lemma 11.9]{s19} which have some omissions. For comparison, we write down all the assumptions. 
\begin{proposition}\mylabel{s19prop}{Proposition \thetheorem}
Let ${\bf K}$ be an AEC with a monster model. Assume ${\bf K}$ is $\mu$-tame, stable in $\mu$ and has $\chi$-local charcter of $\mu$-nonsplitting. Let $\xi\geq\mu^+$. There is $\lambda<h(\xi)$ such that if
\begin{enumerate}
\item ${\bf K}$ is stable in $[\xi,\lambda)$,
\item $\langle M_i:i<\delta\rangle$ is an increasing chain of $\xi$-saturated models;
\item $\cf(\delta)\geq\chi$;
\item Continuity of $\mu$-nonsplitting and of $\xi$-nonsplitting holds,
\end{enumerate}
then $\bigcup_{i<\delta}M_i$ is $\xi$-saturated. 
\end{proposition}
Before proving the proposition, we need to justify that the local character $\chi$ (\ref{chi}), which was defined for $K_\mu$, also applies to $K_\xi$. In other words, we need to show that $K_\xi$ has local character of nonsplitting (at most) $\chi$. (Vasey usually cited this fact as \cite[Section 4]{s3}, by which he should mean an adaptation of \cite[Lemma 4.11]{s3}.)
\begin{lemma}[Local character transfer]\mylabel{chitrans}{Lemma \thetheorem}
If ${\bf K}$ is stable in some $\xi\geq\mu$, then it has $\chi$-local character of $\xi$-nonsplitting.
\end{lemma}
\begin{proof}
Let $\langle M_i:i\leq\delta\rangle$ be u-increasing and continuous in $K_\xi$, $p\in\gs(M_\delta)$. By \ref{gfl}, there is $i<\delta$ such that $p$ does not $\mu$-fork over $M_{i}$. By definition of nonforking, there is $N<_uM_i$ of size $\mu$ such that $p$ does not $\mu$-split over $N$. Suppose $p$ $\xi$-splits over $M_{i}$ then it also $\xi$-splits over $N$. By $\mu$-tameness, it $\mu$-splits over $N$, contradiction.
\end{proof}
\begin{proof}[Proof of \ref{s19prop}]
Let $\delta\geq\chi$ be regular. If $\delta\geq\xi$ we can use a cofinality argument. So we assume $\delta<\xi$. Let $M_\delta\defeq\bigcup_{i<\delta}M_i$ and $N\in K_{<\xi}$, $N\leq M_\delta$, $p\in\gs(N)$. Without loss of generality, we may assume for $i\leq\delta$, $M_i\in K_\xi$: given a saturated $M^*\in K_{\geq\xi^+}$ and some $\tilde{N}\leq M^*$ of size $\leq\xi$, we can close $\tilde{N}$ into a $(\xi,\chi)$-limit $N^*$. By $\xi$-model-homogeneity of $M^*$, we may assume $N^*\leq M^*$. By \ref{chitrans} and \ref{unicor}(2), any $(\xi,\geq\chi)$-limits are saturated, so $N^*$ is saturated. Therefore we can recursively shrink each $M_i$ to a saturated model in $K_\xi$ while still containing the same intersection with $\tilde{N}$.

Extend $p$ to a type in $\gs(M_\delta)$. By \ref{gflfact}, there is $i<\delta$ such that $p$ does not $\mu$-fork over $M_i$. By reindexing assume $i=0$ and let $M_0^0\in K_\mu$ witness the nonforking. Obtain $N_0\in K_\mu$ such that $M_0^0<_uN_0\leq M_0$. Define $\mu'\defeq\mu+\delta$, we build $\langle N_i:1\leq i\leq\delta\rangle$ increasing and continuous in $K_{\mu'}$ such that $N_0\leq N_1\leq N\leq N_\delta$ and for $i\leq\delta$, $N_i\leq M_i$. %We may replace $N$ by $N_\delta$ since we can always show saturation over a bigger model. 
Now we construct
\begin{enumerate}
\item $\langle M_i^*,f_{i,j}:i\leq j<\delta\rangle$ an increasing and continuous directed system;
\item For $i<\delta$, $M_i^*\in K_\xi$, $N_i\leq M_i^*\leq M_i$;
\item For $i<\delta$, $f_{i,i+1}:M_i^*\xrightarrow[N_i]{}M_{i+1}^*$;
\item $M_0^*\defeq M_0$. For $i<\delta$, $f_{i,i+1}[M_i^*]<_uM_{i+1}^*$.
\end{enumerate}

\begin{center}
\begin{tikzcd}
K_\xi  & M_0 \arrow[r]                          & M_1 \arrow[r]                                         & \cdots \arrow[r] & M_\delta \arrow[r]  & \tilde{M}            \\
K_\xi  & M_0^* \arrow[r, "{f_{0,1}}"] \arrow[u] & M_1^* \arrow[u] \arrow[rrr, "{f_{1,\delta}}", dotted] &                  &                     & M_\delta^* \arrow[u] \\
K_\mu' &                                        & N_1 \arrow[u] \arrow[r]                               & \cdots \arrow[r] & N_\delta \arrow[ru] &                      \\
K_\mu  & M_0^0<_u N_0 \arrow[uu] \arrow[ru]     &                                                       &                  &                     &                     
\end{tikzcd}
\end{center}
At limit stage $i<\delta$, take direct limit $M_i^*$ which contains $N_i$. Since $\nr{N_i}<\xi$ and $M_i$ is model-homogeneous, we may assume $M_i^*$ is inside $M_i$. Suppose $M_i^*$ is constructed for some $i<\delta$, obtain the amalgam $M_{i+1}^{**}$ of $M_i^*$ and $N_{i+1}$ over $N_i$. Since $\nr{N_{i+1}}<\xi$ and $M_{i+1}$ is model-homogeneous, we may embed the amalgam into $M_{i+1}$. Call the image of the amalgam $M_{i+1}^*$. After the construction, take one more direct limit to obtain $(M_\delta^*,f_{i,\delta})_{i<\delta}$ (but this time we do not know if $M_\delta^*\leq M_\delta$). By item (4) above, we have that $M_\delta^*$ is a $(\xi,\delta)$-limit, hence saturated.

We will work in a local monster model, namely we find a saturated $\tilde{M}\in K_\xi$ such that \renewcommand{\theenumi}{\alph{enumi}}
\begin{enumerate}
\item $\tilde{M}$ contains $M_\delta$ and $M_\delta^*$;
\item For $i<\delta$, $f_{i,\delta}$ can be extended to $f_{i,\delta}^*\in\oop{Aut}(\tilde{M})$;
\item For $i<\delta$, $f_{i,\delta}^*[N_\delta]\leq M_\delta^*$.
\end{enumerate}
(c) is possible because $M_\delta^*$ is universal over $f_{i,\delta}[M_i^*]$. Finally, we define $N^*\leq M_\delta^*$ of size $\mu'$ containing $\bigcup_{i<\delta}f_{i,\delta}^*[N_\delta]$. By model-homogeneity of $M_\delta^*$, we build $M^{**}\in K_\xi$ saturated such that $N^*\leq M^{**}<_uM_\delta^*$.

By \ref{gfe}, extend $p$ to $q\in\gs(\tilde{M})$ nonforking over $N_0$ (here we need $N_0\in K_\mu$ or else we have to assume more stability). Since $M_\delta^*>_uM^{**}$, we can find $b_\delta\in M_\delta^*$ such that $b_\delta\vDash q\restriction M^{**}$. Since $M_\delta^*$ is a direct limit of the $M_i^*$'s, there is $i<\delta$ such that $f_{i,\delta}(b)=b_\delta$. As $b\in M_i^*\subseteq M_i\leq M_\delta$, it suffices to show that $b\vDash q\restriction (f_{i,\delta}^*)^{-1}[M^{**}]$, because $N\leq N_\delta\leq(f_{i,\delta}^*)^{-1}[N^*]\leq(f_{i,\delta}^*)^{-1}[M^{**}]$. In the following diagram, dotted arrows refer to $\leq$ or $<_u$ between models, while the dashed equal sign is our goal.

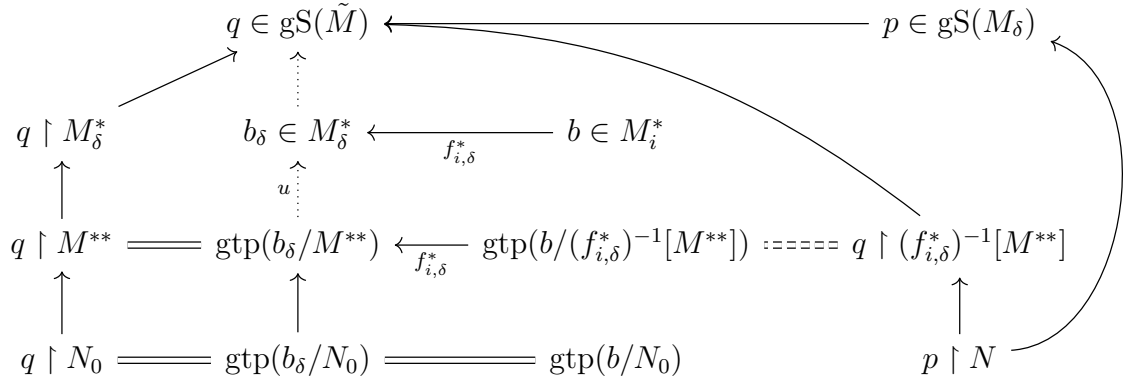
\begin{center}
\begin{tikzcd}
                                                              & q\in \gs(\tilde{M}) &                                                                   &                                                                                                        p\in \gs(M_\delta) \arrow[ll]              \\
q\restriction M_\delta^* \arrow[ru]                           & b_\delta\in M_\delta^* \arrow[u, dotted]                & b\in M_i^* \arrow[l, "{f_{i,\delta}^*}"]                          &                                                                                                                                               \\
q\restriction M^{**} \arrow[u] \arrow[r, Rightarrow, no head] & \gtp(b_\delta/M^{**}) \arrow[u, "u", dotted]                & {\gtp(b/(f_{i,\delta}^*)^{-1}[M^{**}])} \arrow[l, "{f_{i,\delta}^*}"] & {q\restriction (f_{i,\delta}^*)^{-1}[M^{**}]} \arrow[l, Rightarrow, no head, dashed] \arrow[lluu, bend right=18]          \\
q\restriction N_0 \arrow[u] \arrow[r, Rightarrow, no head]    & \gtp(b_\delta/N_0) \arrow[r, Rightarrow, no head] \arrow[u] & \gtp(b/N_0)                                                           &                                                                                                        p\restriction N \arrow[uuu, bend right=79] \arrow[u]
\end{tikzcd}
\end{center} 

Since $q\restriction M^{**}=\gtp(b_\delta/M^{**})$ does not $\mu$-fork over $N_0$ and $f_{i,\delta}^*$ fixes $N_i\geq N_0$, by invariance $\gtp(b/(f_{i,\delta}^*)^{-1}[M^{**}])$ does not $\mu$-fork $N_0$. By monotonicity, $q$ and hence $q\restriction (f_{i,\delta}^*)^{-1}[M^{**}]$ does not $\mu$-fork over $N_0$. By invariance again, $\gtp(b/N_0)=\gtp(b_\delta/N_0)=q\restriction N_0$. By \ref{gfuup}, $q\restriction (f_{i,\delta}^*)^{-1}[M^{**}]=\gtp(b/(f_{i,\delta}^*)^{-1}[M^{**}])$ as desired.
\end{proof}
\begin{remark}
\begin{enumerate}
\item \mylabel{comparermk}{Remark \thetheorem}In \ref{s19prop}, the assumption of stability in $[\xi,\lambda)$ is to guarantee local symmetry from no $\xi$-order property of length $\lambda$. We can relax the stability assumption if we have the stronger assumption of no $\xi$-order property. Namely, if ${\bf K}$ does not have $\xi$-order property of length $\zeta$ where $\zeta>\xi$, then we can simply assume stability in $[\xi,\zeta)$.
\item
We compare our approach with Vasey's. To satisfy hypothesis (4) in \ref{s19result}, he used \ref{ulmfact} which requires $(\xi,\chi)$-symmetry and continuity of nonsplitting \cite[Theorem 11.11(1)]{s19}. Meanwhile he obtained the equivalence of $(\xi,\chi)$-symmetry $\Leftrightarrow$ the increasing union of saturated models of length $\geq\chi$ in $K_{\xi^+}$ is saturated (see \ref{satsymfact}). By \ref{bvfact}, the latter is true for large enough $\xi$. In short, he raised the cardinal threshold while we assumed more stability. More curiously, both our stability assumption and his cardinal threshold are linked to no order property.

A comparison table can be found below. For $\xi\geq\mu$, we abbreviate the increasing union of saturated models of length $\geq\chi$ in $K_{\xi}$ is saturated by ``Union($\xi$)''.
\end{enumerate}
\end{remark}
\begin{table}[h]
\centering
\begin{tabular}{|l||l|}
\hline
Our approach&Vasey's approach  \\ \hline
For $\xi\geq\mu^+$ and & For large enough $\xi$, \\\hline
Enough stability \big($[\mu,h(\xi))$ suffices\big) & $\Rightarrow$ Union($\xi^+$) (\ref{bvfact}) \\ \hline
$\Rightarrow$$(\xi,\chi)$-symmetry (\ref{symcor2}(1))&$\Rightarrow(\xi,\chi)$-symmetry (\ref{satsymfact}) \\ \hline
$\Rightarrow$Saturation of $(\xi,\geq\chi)$-limits & $\Rightarrow$Saturation of $(\xi,\geq\chi)$-limits\\ 
(\ref{unicor}(2))& (\ref{ulmfact})\\\hline
$\Rightarrow$Union($\xi$) (\ref{s19prop})&$\Rightarrow$Union($\xi$) (\ref{s19result}) \\ \hline
\end{tabular}
\end{table}
\begin{obs}
The $[\xi,\lambda)$ stability assumption in \ref{s19prop} can be replaced by $(\xi,\chi)$-symmetry, because we can directly apply \ref{ulmfact} instead of using extra stability to invoke \ref{unicor}. This applies to other results in the paper.
\end{obs}
We now recover two known results with different proofs. The original proof for \cite[Proposition 10.10]{s6} is extremely abstract so we supplement a direct argument. (Here we already assumed a monster model which implies no maximal models everywhere. Alternatively, one can adapt the proof of \cite[Theorem 7.1]{bonext} without using symmetry to transfer no maximal models upward.) On the other hand, since we have generalized the arguments in \cite{vv}, we can specialize them to $\chi=\al$ and recover \cite[Corollary 6.10]{vv} (see below). In their approach, \cite[Theorem 22]{van16b} was cited for the successor case of $\lambda$ and the limit case was proven by inductive hypothesis. We provide a uniform argument to both cases for closure under chains, and fill in the computation of the L\"{o}wenheim-Skolem number for the successor case, which they glossed over.

The following facts do not require continuity of nonsplitting.
\begin{fact}
\begin{enumerate}
\item\mylabel{supfact}{Fact \thetheorem} \cite[Theorem 1]{bkv} Let $\xi\geq\mu$. If ${\bf K}$ is stable in $\xi$, then it is also stable in $\xi^{+n}$ for all $n<\omega$.
\item\cite[Theorem 5.5]{s3} Let $\xi_0\geq\mu$ while $\delta$ be regular, $\langle \xi_i:i<\delta\rangle$ be strictly increasing stability cardinals. If ${\bf K}$ has $\delta$-local character of $\xi_0$-nonsplitting, then $\sup_{i<\delta}\xi_i$ is also a stability cardinal. In particular, if ${\bf K}$ is $\xi$-superstable for some $\xi\geq\mu$, then it is stable in all $\lambda\geq\xi$. 
\end{enumerate}
\end{fact}
\begin{corollary}
\begin{enumerate}
\item \mylabel{chaincor}{Corollary \thetheorem}\emph{\cite[Proposition 10.10]{s6}} Let $\xi\geq\mu$. If ${\bf K}$ is $\xi$-superstable, then it is superstable in all $\zeta\geq\xi$.
\item \emph{\cite[Corollary 6.10]{vv}} Let ${\bf K}$ be $\mu$-superstable and $\xi\geq\mu^+$, then ${\bf K^{\xi\text{-}sat}}$ the class of $\xi$-saturated models in ${\bf K}$ forms an AEC with L\"{o}wenheim-Skolem number $\xi$.
\end{enumerate}
\end{corollary}
\begin{proof}
\begin{enumerate}
\item Combine \ref{supfact}(2) and \ref{chitrans}.
\item By (1) and \ref{weaker}, we have continuity of $\xi$-nonsplitting and stability in $[\xi,\infty)$. By \ref{s19prop}, ${\bf K^{\xi\text{-}sat}}$ is closed under chains. We show that the L\"{o}wenheim-Skolem number is $\xi$: let $A$ be a subset of a $\xi$-saturated model $M$. We need to find a $\xi$-saturated $N\leq M$ of size $\xi+|A|$ containing $A$. 

Consider the case where $\xi$ is regular : then we construct $\langle N_i:i\leq\xi\rangle$ increasing and continuous such that for $1\leq i<\xi$,
\begin{itemize}
\item $N_0$ contains $A$;
\item $N_i\in K_{\xi+|A|}$ is $\xi$-saturated;
\item If $N^*\leq N_i$ is of size less than $\xi$, then $N_{i+1}$ realizes all types over $N^*$.
\end{itemize}
The construction is possible by stability in $\xi+|A|$ (implied by $\mu$-superstability): $M$ is $\xi$-saturated so it has witnesses to all types over $N^*$, but those types can be extended to be over $N_i\in K_{\xi+|A|}$. By stability we can restrict to $(\xi+|A|)$-many witnesses that work for all such $N^*$. Now $N_\xi\leq M$ is $\xi$-saturated by a cofinality argument. Also, it has size $\xi+|A|$. 

For the singular case, write $\xi=\bigcup_{i<\cf(\xi)}\xi_i$ where the $\xi_i$'s form an increasing chain of regular cardinals with $\mu^+\leq\xi_i<\xi$. By the inductive hypothesis that $\oop{LS}({\bf K^{\xi_i\text{-}sat}})=\xi_i$, we can build $\langle N_i:i\leq\cf(\xi)\rangle$ increasing and continuous such that $N_0$ contains $A$, $N_i\in K_{\xi_i+|A|}$ is $\xi_i$-saturated. Since each ${\bf K^{\xi_i\text{-}sat}}$ is closed under chains, $N_\xi$ is $\xi$-saturated and has size $\xi+|A|$.

\end{enumerate}
\end{proof}

It is natural to ask if there are converses to our results. In particular what are the sufficient conditions to ${\bf K}$ having the $\chi$-local character in $K_\xi$ for some $\xi\geq\mu$. \cite[Lemma 4.12]{s19} gave one useful criterion which we adapt below. The original statement did not cover the case $\delta=\xi$ below and such omission affects the rest of his results. In particular \cite[Theorem 4.11]{s19} should only apply to singular $\mu$ there.  Our result covers regular cardinals because we assume stability and continuity of nonsplitting. Only in \cite[Section 11]{s19} did he start to assume continuity of nonsplitting and in \cite[Theorem 12.1]{s19} did he take care of the regular case by under extra assumptions.

We state the full assumptions in the following proposition.
\begin{proposition}\mylabel{localfact}{Proposition \thetheorem}
Let $\mu\geq\ls$. Suppose ${\bf K}$ has a monster model, is $\mu$-tame and stable in some $\xi\geq\mu^+$. Let $\delta<\xi^+$ be regular, $\langle M_i:i\leq\delta\rangle$ be u-increasing and continuous in $K_\xi$ and $p\in\gs(M_\delta)$. There is $i<\delta$ such that $p$ does not $\xi$-split over $M_i$ if one of the following holds:
\begin{enumerate}
\item $\delta=\xi$ (so $\xi$ is regular), ${\bf K}$ has continuity of $\xi$-nonsplitting;
%\item $\delta=\xi$, ${\bf K}$ is stable in $\mu$ and has continuity of $\mu$-nonsplitting;
\item $\delta<\xi$ and $M_\delta$ is $(\mu+\delta)^+$-saturated.
\end{enumerate}
\end{proposition}
\begin{proof}
The first case is by \ref{394prop} (with $\xi$ in place of $\mu$). We consider the second case $\delta<\xi$. Suppose the conclusion is false, then for $i<\delta$, there exist
\begin{enumerate}
\item $N_i^1,N_i^2\in K_\xi$ with $M_i\leq N_i^1,N_i^2\leq M_\delta$;
\item $f_i:N_i^1\cong_{M_i}{}N_i^2$ with $f_i(p\restriction N_i^1)\neq p\restriction N_i^2$;
\item $M_i^1\leq N_i^1$ and $M_i^2\leq N_i^2$ such that $f_i[M_i^1]\cong M_i^2$ and $f_i(p\restriction M_i^1)\neq p\restriction M_i^2$.
\end{enumerate}
Let $N\leq M_\delta$ of size $\mu+\delta$ containing $M_i^1$ and $M_i^2$ for all $i<\delta$. Since $M_\delta$ is $(\mu+\delta)^+$-saturated, there is $b\in|M_\delta|$ realizing $p\restriction N$. Then there is $i<\delta$ such that $b\in|M_i|$. Since $f_i$ fixes $M_i$, it also fixes $b$. Thus $$f_i(p\restriction M_i^1)=\gtp(f_i(b)/M_i^2)=\gtp(b/M_i^2)=p\restriction M_i^2,$$ contradicting item (3) above.
\end{proof}
\begin{corollary}\mylabel{localcor}{Corollary \thetheorem}
Suppose $\xi\geq\mu^+$ and $\delta<\xi^+$ be regular. If ${\bf K}$ is stable in $\xi$, has continuity of $\xi$-nonsplitting and has unique $(\xi,\geq\delta)$-limit models, then it has $\delta$-local character in $K_\xi$. If in addition $K_\xi$ has unique limit models, then it is $\xi$-superstable.
\end{corollary}
\begin{proof}
Let $\delta'\geq\delta$ be regular and $\langle M_i:i\leq\delta'\rangle\subseteq K_\xi$ be u-increasing and continuous, $p\in\gs(M_{\delta'})$. By the proof of \ref{unicor}(2), $M_{\delta'}$ is saturated. By \ref{localfact}, there is $i<\delta'$ such that $p$ does not $\xi$-split over $M_i$.
\end{proof}
\begin{remark}
As before, our result is local. \cite[Theorem 3.18]{GV} proved a similar result which is eventual: they managed to guarantee superstability after $\beth_\omega(\chi_0)$ where ${\bf K}$ has no order property of length $\chi_0$.
\end{remark}
Vasey \cite[Fact 11.6]{s19} also made another observation that connects saturated models and symmetry. In the original statement, he omitted writing continuity of nonsplitting in the hypothesis and did not give a proof sketch, so we give more details here  (\ref{assum1} applies). As in the discussion before \ref{musymdef}, we consider the tail of regular cardinals $\delta'\geq\delta$ in place of a fixed $\delta'=\delta$ to match our notations.
\begin{fact}\mylabel{satsymfact}{Fact \thetheorem}
Let $\delta<\mu^+$ be regular. If for any $\delta'\in[\delta,\mu^+)$ regular, any $\langle M_i:i<\delta'\rangle$ increasing chain of saturated models in $K_{\mu^+}$ has a saturated union, then ${\bf K}$ has $(\mu,\delta)$-symmetry.
\end{fact}
\begin{proof}
In \cite[Theorem 2]{van16a}, it was shown that if the above fact holds for any $\delta<\mu^+$, then any reduced tower is continuous at all $\delta<\mu^+$. We can localize this argument to show that if the above fact holds for a specific $\delta<\mu^+$, then any reduced tower is continuous at $\geq\delta$.  By \cite[Proposition 19]{bvulm}, ${\bf K}$ has $(\mu,\delta)$-symmetry.
\end{proof}
\begin{corollary}\mylabel{uufact}{Corollary \thetheorem}
Let $\delta<\mu^+$ be regular. If for any $\delta'\in[\delta,\mu^+)$ regular, any $\langle M_i:i<\delta'\rangle$ increasing chain of saturated models in $K_{\mu^+}$ has a saturated union, then ${\bf K}$ has uniqueness of $(\mu,\geq\delta)$-limit models.
\end{corollary}
\begin{proof}
Combine \ref{satsymfact} and \ref{ulmfact}.
\end{proof}
\begin{ques}
Is there an analog of \ref{satsymfact} and \ref{uufact} where ``$\mu^+$'' is replaced by a general $\xi\geq\mu^+$?
\end{ques}

We look at superlimits and solvability before ending this section. The following localizes \cite[Definition 2.1]{ss}, which is more natural than \cite[Definition 6.2]{s19}.
\begin{definition}\mylabel{suplimdef}{Definition \thetheorem}
Let $\xi\geq\mu$. $M\in K_\xi$ is a \emph{$\chi$-superlimit} if $M$ is universal in $K_\xi$, not maximal, and for any regular $\delta$ with $\chi\leq\delta<\xi^+$, $\langle M_i:i<\delta\rangle$ increasing such that $M_i\cong M$ for all $i<\delta$, then $\bigcup_{i<\delta}M_i\cong M$. $M$ is called a \emph{superlimit} if it is a $\al$-superlimit.
\end{definition}
\begin{proposition}\mylabel{supexist}{Proposition \thetheorem}
Let ${\bf K}$ have continuity of $\xi$-nonsplitting for some $\xi\geq\mu^+$. There is $\lambda<h(\xi)$ such that if ${\bf K}$ is stable in $[\xi,\lambda)$, then it has a saturated $\chi$-superlimit in $K_\xi$.
\end{proposition}
\begin{proof}
By \ref{unicor}(2) and \ref{chitrans}, any $(\xi,\geq\chi)$-limit $M$ is saturated (hence universal in $K_\xi$). Let $\delta$ be regular, $\chi\leq\delta<\xi^+$, $\langle M_i:i<\delta\rangle$ increasing such that $M_i\cong M$ for all $i<\delta$. Then all $M_i$ are saturated in $K_\xi$. By \ref{s19prop}, $\bigcup_{i<\delta}M_i$ is also saturated, hence isomorphic to $M$.
\end{proof}
\begin{remark}
The specific $\chi$-superlimit built above is saturated. Under the same assumptions, it is true for all $\chi$-superlimits (\ref{supsat}).
\end{remark}

The following connects superlimit models with \emph{solvability} (see \cite[Definition 2.17]{GV} for a definition).
\begin{fact}\cite[Lemma 2.19]{GV}\mylabel{solvfact}{Fact \thetheorem}
Let $\lambda\geq\xi$. The following are equivalent:
\begin{enumerate}
\item ${\bf K}$ is $(\lambda,\xi)$-solvable.
\item There exists an AEC ${\bf K'}$ in ${\oop{L}({\bf K'})}\supseteq\llk$ such that $\oop{LS}({\bf K'})\leq\xi$, ${\bf K'}$ has arbitarily large models and for any $M\in K'_\lambda$, $M\restriction \llk$ is a superlimit in ${\bf K}$.
\end{enumerate}
\end{fact}
In \cite[Theorem 4.9]{GV}, they showed that $(\lambda,\xi)$-solvability is \emph{eventually} (in $\lambda$) equivalent to other criteria of superstability (modulo a jump of $\beth_{\omega+2}$). Also, $\lambda$ is required to be greater than $\xi$. We propose that a better formulation of superstability which has $\lambda=\xi$. The case $\lambda>\xi$ should be a stronger condition because it allows downward transfer (see \cite[Corollary 5.1]{s17} for more development on this). Our result proceeds with a series of lemmas.

The next lemma generalizes \cite[Fact 2.8(5)]{GV} (which is based on \cite{dru}).
\begin{lemma}\mylabel{chainprop}{Lemma \thetheorem}
Let $\xi\geq\mu^+$ and $M$ be a saturated model in $K_\xi$. $M$ is a $\chi$-superlimit iff for any regular $\delta$ with $\chi\leq\delta<\xi^+$, any increasing chain of saturated models in $K_\xi$ of length $\delta$ has a saturated union.
\end{lemma}
\begin{proof}
Immediate from the definition of a $\chi$-superlimit. Notice that we need $\delta<\xi^+$ to make sure that the chain of saturated models have a union in $K_\xi$.
\end{proof}

The following lemma generalizes \cite[Theorem 2.3.11]{dru}.
\begin{lemma}\mylabel{supsat}{Lemma \thetheorem}
Let $\xi>\ls$. If $M$ is a $\chi$-superlimit in $K_\xi$, then $M$ is saturated.
\end{lemma}
\begin{proof}
We show that $M$ is a $(\xi,\delta)$-limit for regular $\delta\in[\chi,\xi^+)$. If done, the argument in \ref{unicor}(2) shows that it is saturated. Construct $\langle M_i,N_i:i<\delta\rangle$ in $K_\xi$ such that $M_0\defeq M\cong M_i<_uN_i<M_{i+1}$ for $i<\delta$. Suppose $N_i$ is constructed, by universality $N_i$ embeds inside $M$ so we can build $M_{i+1}$, an isomorphic copy of $M$ over $N_i$. To construct $M_i$ for limit $i$, we embed the union of previous $N_i$ inside $M$ and repeat the above process. By the property of a $\chi$-superlimit, $M\cong\bigcup_{i<\delta}M_i=\bigcup_{i<\delta}N_i$ which is a $(\xi,\delta)$-limit.
\end{proof}
\begin{proposition}
\mylabel{supsolv}{Proposition \thetheorem}%Let $\xi\geq\mu^+$, ${\bf K}$ have continuity of $\xi$-nonsplitting and be stable in $\xi$. ${\bf K}$ is $\xi$-superstable iff it is $(\xi,\xi)$-solvable.
If $\mu>\ls$ and ${\bf K}$ is $(<\mu)$-tame, then it is $\mu$-superstable iff it is $(\mu^+,\mu^+)$-solvable.
\end{proposition}
\begin{proof}
Suppose ${\bf K}$ is $\mu$-superstable. By \ref{supsat} with $\xi=\mu^+$, superlimits in $K_\xi$ are saturated. By \ref{chaincor}(2), $\xi$-saturated models are closed under chains. By \ref{chainprop}, saturated models in $K_\xi$ are superlimits. Therefore, saturated models and superlimits coincide in $K_\xi$. By \ref{solvfact}, we can define $\oop{L}({\bf K'})\defeq\llk$ and ${\bf K'}$ to be the class of $\xi$-saturated models. By \ref{chaincor}(2) again, it is an AEC with $\oop{LS}({\bf K'})=\xi$.

Suppose ${\bf K}$ is $(\mu^+,\mu^+)$-solvable. By \ref{supsat} there is a saturated superlimit in $K_{\mu^+}$, which witnesses the union of saturated models in $K_{\mu^+}$ is $\mu^+$-saturated. By \ref{uufact}, it has uniqueness of limit models in $K_\mu$. By $(<\mu)$-tameness and the proof of \ref{localcor} (replace ``$\xi$'' there by $\mu$ and ``$\mu^+$'' there by ${\ls}^+$), it is $\mu$-superstable.
\end{proof}
\begin{remark}
One might want to generalize the argument to strictly stable AECs. In that case the statement of \ref{solvfact}(2) should naturally be for a $\chi$-AEC instead of an AEC, but we do not know how to prove that saturated models are closed under $\chi$-directed systems (a similar obstacle is in \cite[Remark 2.3(4)]{muaec}). On top of that, the equivalence in \ref{solvfact} is not clear in that case because we do not have a first-order presentation theorem on $\chi$-AECs to extract an Ehrenfeucht-Mostowski blueprint (but we do have a $(<\mu)$-ary presentation theorem, see \cite[Theorem 3.2]{muaec} or \cite[Theorem 5.6]{leung2}).
\end{remark}
\section{Stability in a tail and U-rank}\label{ufsec7}
In this section we look at two characterizations of superstability. For convenience we follow \cite[Section 4]{s19} to define some cardinals:
\begin{definition}\mylabel{chap5ques2}{Definition \thetheorem}
\begin{enumerate}
\item $\lambda({\bf K})$ stands for the first stability cardinal above $\ls$. 
\item $\chi({\bf K})$ stands for the least regular cardinal $\delta$ such that ${\bf K}$ has $\delta$-local character of $\xi$-nonsplitting for some stability cardinal $\xi\geq\ls$. 
\item $\lambda'({\bf K})$ stands for the minimum stability cardinal $\xi$ such that for any stability cardinal $\xi'\geq\xi$, ${\bf K}$ has $\chi({\bf K})$-local character of $\xi'$-nonsplitting.
\end{enumerate}
\end{definition}
\begin{obs}
\begin{enumerate}
\item By \ref{assum1}, $\lambda({\bf K})\leq\mu$. 
\item By \ref{chi} (see also the remark after it), $\chi({\bf K})\leq\chi$. 
\item By \ref{chitrans}, we can equivalently define $\lambda'({\bf K})$ as the minimum stability cardinal $\xi$ such that ${\bf K}$ has $\chi({\bf K})$-local character of $\xi$-nonsplitting. 
\item ${\bf K}$ is eventually superstable ($\xi$-superstable for large enough $\xi$) iff $\chi({\bf K})=\al$.
\end{enumerate}
\end{obs}
Currently we do not have a nice bound of $\lambda'({\bf K})$ so the cardinal threshold might be very high if we invoke $\lambda'({\bf K})$ or $\chi({\bf K})$. Vasey built upon \cite{sh394} and spent several sections to derive:
\begin{fact}\cite[Theorem 11.3(2)]{s19}\mylabel{hanffact}{Fact \thetheorem}
Suppose ${\bf K}$ has continuity of $\xi$-nonsplitting for all stability cardinal $\xi$, then $\lambda'({\bf K})<h(\lambda({\bf K}))$. 
\end{fact}

We can now state Vasey's characterization that superstability is equivalent to stability in a tail of cardinals. Since continuity of $\mu$-nonsplitting is not assumed there, item (1) only holds for singular $\xi$. Also, the original formulation wrote $\lambda'({\bf K})$ instead of $(\lambda'({\bf K}))^+$ but the proof did not go through.
\begin{fact}\mylabel{splsupfact}{Fact \thetheorem}
Let ${\bf K}$ be $\ls$-tame with a monster model.
\begin{enumerate}
\item \cite[Corollary 4.14]{s19} Let $\chi_1$ as in \ref{bvfact}, $\xi\geq(\lambda'({\bf K}))^++\chi_1$ be singular, ${\bf K}$ be stable in unboundedly many cardinal $<\xi$. ${\bf K}$ is stable in $\xi$ iff $\cf(\xi)\geq\chi({\bf K})$.
\item \cite[Corollary 4.24]{s19} $\chi({\bf K})=\al$ iff ${\bf K}$ is stable in a tail of cardinals.
\end{enumerate}
\end{fact} 
We prove a simpler and local analog to \ref{splsupfact}. Rather than looking at the whole tail of cardinals (more accurately the class of singular cardinals with all possible cofinalities) after a potentially high threshold, we directly look for the next $\omega+1$ many cardinals of $\mu$ and verify that ${\bf K}$ has enough stability, continuity of nonsplitting and symmetry in those cardinals. Symmetry will be guaranteed by more stability.

\begin{proposition}\mylabel{cfcal}{Proposition \thetheorem}
There is $\lambda<h(\mu^{+\omega})$ such that if ${\bf K}$ is stable in $[\mu,\lambda)$ and has continuity of $\mu^{+\omega}$-nonsplitting, then it is $\mu^{+\omega}$-superstable.
\end{proposition}
\begin{proof}
%The forward direction does not depend on $\lambda$ and is by \ref{chaincor}(1) and \ref{weaker}(1). For the backward direction, 
Obtain $\lambda$ from \ref{unicor}(2) and suppose ${\bf K}$ is stable in $[\mu,\lambda)$ and has continuity of $\mu^{+\omega}$. The conclusion of \ref{unicor}(2) (which uses stability in $\mu^{+\omega}$ and continuity of $\mu^{+\omega}$-nonsplitting) gives a saturated model $M$ of size $\mu^{+\omega}$. We show that is a $(\mu^{+\omega},\omega)$-limit: by stability in $[\mu,\mu^{+\omega})$, build $\langle M_n:n\leq\omega\rangle\subseteq K_{<\mu^{+\omega}}$ u-increasing and continuous such that for $n<\omega$, $M_n\in K_{\mu^{+n}}$ and $M_\omega=M$. On the other hand, by stability in $\mu^{+\omega}$, build $\langle N_i:i\leq\omega\rangle\subseteq K_{\mu^{+\omega}}$ u-increasing and continuous such that $M_0\leq N_0$. By a back-and-forth argument, $M\cong_{M_0}N_\omega$ and the latter is a $(\mu^{+\omega},\omega)$-limit. By uniqueness of limit models of the same cofinality, any $(\mu^{+\omega},\omega)$-limit is saturated.

By \ref{localfact}(2) where $\xi=\mu^{+\omega}$, $\delta=\al$, ${\bf K}$ has $\al$-local character of $\mu^{+\omega}$-nonsplitting. Together with stability in $\mu^{+\omega}$, we know that ${\bf K}$ is superstable in $\mu^{+\omega}$.
\end{proof}
We state a more general form of the above proposition:
\begin{corollary}\mylabel{cfcalcor}{Corollary \thetheorem}
Let $\delta$ be a regular cardinal. There is $\lambda<h(\mu^{+\delta})$ such that if ${\bf K}$ is stable in $[\mu,\lambda)$ and has continuity of $\mu^{+\delta}$-nonsplitting, then it has $\delta$-local character of $\mu^{+\delta}$-nonsplitting. Stability in $[\mu,\lambda)$ can be replaced by stability in $[\mu^{+\delta},\lambda)$ and unboundedly many cardinals below $\mu^{+\delta}$.
\end{corollary}
\begin{proof}
Replace ``$\omega$'' by $\delta$ in \ref{cfcal}. Notice that unboundedly stability many cardinals below $\mu^{+\delta}$ are sufficient to build $\langle M_i:i<\delta\rangle\subseteq K_{<\mu^{+\delta}}$ u-increasing.
\end{proof}
\begin{remark}
\begin{enumerate}
\item \mylabel{cfcalrmk}{Remark \thetheorem}A missing case of \ref{cfcal} is perhaps the regular cardinal $\al$. In \cite[Theorem 2]{bkv}, it was shown that if ${\bf K}$ has $\omega$-locality, $\aleph_0$-tameness and stability in $\al$, then ${\bf K}$ is stable everywhere. The original proof used a tree argument of height $\omega$. We provide an alternative proof using our general tools: by $\omega$-locality and \ref{weaker}(2), ${\bf K}$ has continuity of $\al$-nonsplitting. By \ref{394prop}, ${\bf K}$ has $\al$-local character of $\al$-nonsplitting. By \ref{chaincor}(1), it is (super)stable everywhere.
\item Our proof strategy of \ref{cfcal} is similar to that of \cite[Theorem 4.11]{s19} but we use different tools. Both assume stability in $\mu^{+\omega}$ and unboundedly many cardinals in $\mu^{+\omega}$. To obtain a saturated model, Vasey raised the threshold of $\mu$ so that the union of $\mu^{+n}$-saturated models is $\mu^{+n}$-saturated (see \ref{bvfact}). Then he used \cite[Theorem 4.13]{s19} that models in $K_{\mu^{+\omega}}$ can be closed to a $\mu^{+n}$-saturated model. These two give a saturated model in $K_{\mu^{+\omega}}$. In contrast, we bypass such gap by using the uniqueness of long enough limit models in $K_{\mu^{+\omega}}$, this immediately gives us a saturated model in $K_{\mu^{+\omega}}$. After that, Vasey and our approaches converge: the saturated model is a $(\mu^{+\omega},\omega)$-limit and \ref{localfact} gives $\al$-local character of $\mu^{+\omega}$-nonsplitting.
\end{enumerate}
\end{remark}%ch

\begin{ques}
\begin{enumerate}
\item \mylabel{lkques}{Question \thetheorem} Perhaps under extra assumptions, is it possible to obtain a tighter bound of $\lambda'({\bf K})$ in terms of $\lambda({\bf K})$ than in \ref{hanffact}?
\item Let $\xi_1,\xi_2$ be stability cardinals. Is there any relationship between continuity of $\xi_1$-nonsplitting and continuity of $\xi_2$-nonsplitting? Similarly, can one say anything about continuity of $\xi_1$-nonsplitting if for unboundedly many stability cardinal $\xi<\xi_1$, ${\bf K}$ has continuity of $\xi$-nonsplitting? A positive answer might help improve \ref{cfcal}.
\end{enumerate}
\end{ques}

In \cite[Section 7]{BG}, Boney and Grossberg developed a $U$-rank for an independence relation over types of arbitrary length. Until \ref{BGfact2}, we specify that we only need an independence relation over 1-types for the proofs to go through.
\begin{definition}\cite[Definition 7.2]{BG}
Let ${\bf K}$ have a monster model and an independence relation over types of length one. $U$ is a class function that maps each Galois type (of length one) in the monster model to an ordinal or $\infty$, such that for any $M\in K$, $p\in\gs(M)$, 
\begin{enumerate}
\item $U(p)\geq0$;
\item For limit ordinal $\alpha$, $U(p)\geq\alpha$ iff $U(p)\geq\beta$ for all $\beta<\alpha$;
\item For an ordinal $\beta$, $U(p)\geq\beta+1$ iff there is $M'\geq M$, $\nr{M'}=\nr{M}$ and $p'\in\gs(M')$ such that $p'$ is a forking (in the sense of the given independence relation) extension of $p$ and $U(p')\geq\beta$;
\item For an ordinal $\alpha$, $U(p)=\alpha$ iff $U(p)\geq\alpha$ but $U(p)\not\geq\alpha+1$;
\item $U(p)=\infty$ iff $U(p)\geq\alpha$ for all ordinals $\alpha$.
\end{enumerate}
\end{definition}
Through a series of lemmas, they managed to obtain the following fact (\ref{assum1} is not needed). 
\begin{fact}\cite[Theorem 7.9]{BG}\mylabel{BGfact}{Fact \thetheorem}
Let ${\bf K}$ have a monster model and an independence relation over types of length one. Suppose the independence relation satisfies invariance and monotonicity.
Let $M\in K$ and $p\in\gs(M)$. The following are equivalent:
\begin{enumerate}
\item  $U(p)=\infty$;
\item There is $\langle p_n:n<\omega\rangle$ such that $p_0=p$ and for $n<\omega$, the domain of $p_n$ has size $\nr{M}$, and $p_{n+1}$ is a forking extension of $p_n$. 
\end{enumerate}
\end{fact}
The original proof proceeds with a lemma followed by the theorem statement. Since the proof of the lemma omitted some details, and that the lemma and the theorem made reference to each other, we straighten the proof as follows:
\begin{lemma}\mylabel{BGlem1}{Lemma \thetheorem}
(2)$\Rightarrow$(1) holds in \ref{BGfact}.
\end{lemma}
\begin{proof}
By induction on each ordinal $\alpha$, we show that for each $\alpha$, for each $n<\omega$, $U(p_n)\geq\alpha$. The base case $\alpha=0$ is by the definition of $U$. The limit case follows from the inductive hypothesis. Suppose we have proven the case $\alpha$, then for each $n<\omega$, inductive hypothesis gives $U(p_{n+1})\geq \alpha$. By the definition of $U$, $U(p_n)\geq\alpha+1$.
\end{proof}
\begin{lemma}\mylabel{BGlem2}{Lemma \thetheorem}
Let ${\bf K}$ have a monster model and an independence relation over types of length one. Suppose the independence relation satisfies invariance and monotonicity. Let $\lambda\geq\ls$. There is an ordinal $\alpha_\lambda<(2^\lambda)^+$ such that for $M\in K_\lambda$, $p\in\gs(M)$, if $U(p)\geq\alpha_\lambda$ then $U(p)=\infty$.
\end{lemma}
\begin{proof}
By invariance, there are at most $2^\lambda$ many $U$-ranks of types over models of size $\lambda$. It suffices to show that there is no gap in the $U$-rank: if $\beta$ is an ordinal, $N\in K_\lambda$, $q\in\gs(N)$ with $\beta<U(q)<\infty$, then there is a forking extension $q'$ of $q$ (with domain of size $\lambda$) such that $U(q')=\beta$. Otherwise pick a counterexample $q\in\gs(N)$. Since $U(q)\geq\beta+1$, there is a forking extension $q_1$ of $q$ such that $U(q_1)\geq\beta$. As $U(q_1)$ cannot be $\beta$, $U(q_1)\geq\beta+1$. Using monotonicity of forking, we can inductively build $\langle q_n:n<\omega\rangle$ with $q_0\defeq q$ and for $n<\omega$, $q_{n+1}$ is a forking extension of $q_n$. By \ref{BGlem1}, $U(q_0)=U(q)=\infty$, contradicting the assumption on $U(q)$.
\end{proof}
\begin{lemma}\mylabel{BGlem3}{Lemma \thetheorem}
Let ${\bf K}$ have a monster model and an independence relation over types of length one. Suppose the independence relation satisfies invariance and monotonicity.
Then (1)$\Rightarrow$(2) in \ref{BGfact} holds.
\end{lemma}
\begin{proof}
Let $\lambda=\nr{M}$, $\alpha_\lambda$ as in \ref{BGlem2} and $p_0\defeq p$. Define $\langle p_n:n<\omega\rangle$ inductively such that $U(p_n)=\infty$. The base case is by assumption on $p$. Suppose $p_n$ is constructed with $U(p_n)=\infty$, then in particular $U(p_n)\geq\alpha_\lambda+1$. By definition of $U$, there is a forking extension $p_{n+1}$ of $p_n$ (with domain of size $\lambda$) such that $U(p_{n+1})\geq\alpha_\lambda$. By \ref{BGlem2} again, $U(p_{n+1})=\infty$.
\end{proof}
\begin{proof}[Proof of \ref{BGfact}]
Combine \ref{BGlem1} and \ref{BGlem3}.
\end{proof}
We have now arrived at an alternative characterization of superstability. At the end of \cite[Section 6]{GV}, they suggested the use of coheir and show that superstability implies bounded $U$-rank. Since we cannot verify the claim, we use instead $\mu$-nonforking as the independence relation to characterize superstability as bounded $U$-rank for limit models in $K_\mu$.
\begin{corollary}\mylabel{urankcor}{Corollary \thetheorem}
Under \ref{assum1}, restrict $\mu$-nonforking to limit models in $K_\mu$ ordered by $\leq_u$. Then ${\bf K}$ is $\mu$-superstable iff $U(p)<\infty$ for all $p\in\gs(M)$ and limit model $M\in K_\mu$.
\end{corollary}
\begin{proof}
By \ref{BGfact}, we need to show $\mu$-superstability is equivalent to the negation of criterion (2) there. By continuity of $\mu$-nonforking (\ref{gfc}) and the proof of \ref{bgvvlem}, it suffices to prove that $\mu$-superstability is equivalent to $\mu$-nonforking having local character $\al$ (under $AP$ it is always possible to extend an omega-chain of types). 
The forward direction is given by \ref{gfl} and the backward direction is given by \ref{gfe}, \ref{gfu} and \ref{getloc}.
\end{proof}

We look at one more result of $U$-rank, which shows the equivalence of being a nonforking extension and having the same $U$-rank (\ref{BGfact2}). The extra assumption of $\ls$-witness property for singletons was pointed out by \cite[Lemma 8.8]{GM} to allow the proof of monotonicity of $U$-rank \cite[Lemma 7.3]{BG} to go through. We will adapt their definition of $\ls$-witness property for singletons because our nonforking is originally defined for model-domains while their independence relations assume set-domains (another approach is perhaps to work in the closure (\ref{clodef}) of nonforking, but we will not pursue it here).  %; see \cite[Definition 3.3]{BG} and \cite[Definition 2.11]{GM} for relevant definitions.
\begin{definition}
\begin{enumerate}
\item Let $\lambda$ be a cardinal. An independence relation $\fork$ has the $\lambda$-witness property if the following holds: let $a$ be a singleton and $M,N\in K$. If for any $M'$ with $M\leq M'\leq N$, $\nr{M'}\leq\nr{M}+\lambda$, we have $\fk{a}{M}{M'}$, then $\fk{a}{M}{N}$.
\item An independence relation satisfies left transitivity if the following holds: let $A$ be a set, $M_0\leq M_1\leq N$ with $\fk{A}{M_1}{N}$ and $\fk{M_1}{M_0}{N}$, then $\fk{A}{M_0}{N}$.
\end{enumerate}
\end{definition}
\begin{fact}\cite[Theorem 7.7]{BG}\mylabel{BGfact2}{Fact \thetheorem}
Let ${\bf K}$ have a monster model and an independence relation over types of arbitrary length. Suppose the independence relation satisfies: invariance, monotonicity, left transitivity, existence, extension, uniqueness, symmetry and $\ls$-witness property for singletons.  For any $p\in\gs(M_0)$, any $q\in\gs(M_1)$ extending $p$ such that both $U(p),U(q)<\infty$, then 
$$U(p)=U(q)\Leftrightarrow q\text{ is a nonforking extension of }p$$
\end{fact}

We notice a gap in \cite[Lemma 7.6]{BG} which \ref{BGfact2} depends on (readers can skip to \ref{BGvarfact} if they simply use \ref{BGfact2} as a blackbox; we will also give an alternative proof that does not depend on the lemma). As usual, their definition of independence relations assume that the domain contains the base: if we write $\fk{A}{M}{N}$, we assume $M\leq N$. In the proof of \cite[Lemma 7.6]{BG}, they applied monotonicity to obtain $\fk{N_2c}{\bar{N_0}}{N_1}$. However, $\bar{N_0}\not\leq N_1$ because $c\in\bar{N_0}-N_1$ might happen. We will rewrite the proof in \ref{BGrewrite} using the idea of a closure of an independence relation, and drawing results from \cite{BGKV}.
\begin{definition}\cite[Definition 3.4]{BGKV}\mylabel{clodef}{Definition \thetheorem}
$\bar{\fork}$ is a closure of an independence relation $\fork$ if it satisfies the following properties:
\begin{enumerate}
\item $\bar{\fork}$ is defined on triples of the form $(A,M,B)$ where $M\in K$, $A$ and $B$ are sets of elements. We allow $M\not\subseteq B$. 
\item Invariance: if $f\in\oop{Aut}(\mn)$ and $\fkc{A}{M}{B}$, then $\fkc{f[A]}{f[M]}{f[B]}$;
\item Monotonicity: if $\fkc{A}{M}{B}$, $A'\subseteq A$, $B'\subseteq B$, then $\fkc{A'}{M}{B'}$;
\item Base monotonicity: if $\fkc{A}{M}{B}$ and $M\leq M'\subseteq M\cup B$, then $\fkc{A}{M'}{B}$.
\end{enumerate}
The minimal closure of $\fork$ (which is the smallest closure of $\fork$) is defined by: $\fkc{A}{M}{C}$ iff there is $N\geq M$, $N\supseteq C$ such that $\fk{A}{M}{N}$.
\end{definition}
We quote the following lemma without proof.
\begin{lemma}\emph{\cite[Lemmas 5.1, 5.3, 5.4]{BGKV}}\mylabel{closlem}{Lemma \thetheorem}
Let $\fork$ be an independence relation for types of arbitrary length, $\bar{\fork}$ be the minimal closure of $\fork$. 
\begin{enumerate}
\item $\fork$ has symmetry iff $\bar{\fork}$ has symmetry.
\item Suppose $\fork$ has extension. Then $\fork$ has left transitivity iff $\bar{\fork}$ does.
\item $\fork$ has extension iff $\bar{\fork}$ has extension.
\end{enumerate}
\end{lemma}
\begin{proposition}\mylabel{BGrewrite}{Proposition \thetheorem}
Under the same hypothesis as \ref{BGfact}, let $N_0\leq N_1\leq \bar{N_1}$; $N_0\leq \bar{N_0}\leq \bar{N_1}$; $N_0\leq N_2$; $c\in|\bar{N_0}|$. If 
$$\fk{N_1}{N_0}{\bar{N_0}}\text{ and }\fk{N_2}{\bar{N_0}}{\bar{N_1}}$$
then there is some $N_3$ extending both $N_1$ and $N_2$ such that 
$$\fk{c}{N_2}{N_3}.$$
\end{proposition}
\begin{proof}We write $\bar{\fork}$ to mean the minmal closure of the given independence relation $\fork$. By symmetry twice on $\fk{N_2}{\bar{N_0}}{\bar{N_1}}$, there is $\bar{N_2}$ containing $c$ and extending $\bar{N_0},N_2$ such that $\fk{\bar{N_2}}{\bar{N_0}}{\bar{N_1}}$. By definition of the minimal closure, $$\fkc{\bar{N_2}}{\bar{N_0}}{{N_1}}.$$ On the other hand, by symmetry (and monotonicity) on $\fk{N_1}{N_0}{\bar{N_0}}$, $\fk{\bar{N_0}}{N_0}{N_1}$. Then $\fkc{\bar{N_0}}{N_0}{N_1}$. Applying \ref{closlem}(2) to the last two closure independence, we have $\fkc{N_2c}{N_0}{N_1}$. By \ref{closlem}(1), there is $N_3'\geq N_2$ and containing $c$ such that $\fkc{N_1}{N_0}{N_3'}$. By definition of the minimal closure, $\fk{N_1}{N_0}{N_3'}$. (Here we return to the original proof.) By base monotonicity, $\fk{N_1}{N_2}{N_3'}$. By symmetry, there is $N_3$ extending $N_1$ and $N_2$ such that $\fk{N_3'}{N_2}{N_3}$. By monotonicity, $\fk{c}{N_2}{N_3}$ as desired.
\end{proof}

Back to \ref{BGfact2}, we would like to know if there are any examples of independence relations that satisfy its hypotheses. The approach in \cite{BG} is to consider coheir \cite[Definition 3.2]{BG}, assuming tameness, shortness, no weak order property and that coheir satisfies extension. More developments of coheir can be found in \cite{s6} but the framework there is too abstract to handle. 

Another natural candidate is $\mu$-nonforking. One obstacle is that the hypotheses in \ref{BGfact2} require the independence relation to be over types of arbitrary length, while we have defined it for singletons only. Another obstacles is that if we extend our frame to longer types, we might not necessarily guarantee type-fullness (existence holds for all nonalgebraic types), so we cannot invoke \ref{BGfact2}. To resolve these, we use the following fact to extend our frame to types of arbitrary length, while acknowledging that the new frame might not be type-full. Then we give an alternative proof to \ref{BGfact2} that does not use existence.

We state the full assumptions of the following facts.
\begin{fact}\mylabel{BGvarfact}{Fact \thetheorem}
Let ${\bf K}$ have a monster model, $\lambda\geq\ls$.
\begin{enumerate}
\item {\cite[Theorem 1.1]{bv}} Suppose ${\bf K}$ is $\lambda$-tame and there is a good $(\geq\lambda)$-frame perhaps except the symmetry property. Then the frame can be extended to a (perhaps non-type-full) good frame for types of arbitrary length and satisfying symmetry.
\item {\cite[Lemma 5.9]{BGKV}} Let $\fork$ be an independence relation for types of arbitrary length. Suppose $\fork$ satisfies symmetry and right transitivity, then it satisfies left transitivity. 
%\item {\cite[Lemma 5.9]{s3}} Suppose ${\bf K}$ is superstable in $\lambda$. Then there is a good $(\geq\lambda^+)$-frame over the $\lambda^+$-saturated models.
\end{enumerate}
\end{fact}
\begin{remark}
\begin{enumerate}
\item \ref{BGvarfact}(1) is achieved by \emph{independent sequences}. If we simply build nonforking from nonsplitting for longer types, then some of the results in this paper do not generalize (for example stability of $\mu$-types in $K_\mu$ immediately fails). One would need extra assumptions (say shortness) and to build the frame in higher cardinals. See also \cite[Appendix A]{s8}.
\item Another known approach to get a type-full frame for longer types is via Shelah's NF. Vasey \cite[Sections 11, 12]{s6} showed that with shortness (which we do not assume in this paper), one can extend a nice enough frame by NF, which is type-full.
\end{enumerate}
\end{remark}
Under $\mu$-superstability, we can derive an independence relation that satisfies all the hypotheses of \ref{BGfact2} except for existence for longer types. We will use \ref{assum1}.
\begin{proposition}\mylabel{BGfactprop}{Proposition \thetheorem}
Let ${\bf K}$ be $\mu$-superstable. Let ${\bf K'}$ be the AEC of the limit models in $K_{\geq\mu}$ ordered by $\leq_u$. Then $\mu$-nonforking restricted to ${\bf K'}$ can be extended to a (perhaps non-type-full) good frame for types of arbitrary length. Also it satisfies left transitivity and $\mu$-witness property for singletons.
\end{proposition}
\begin{proof}
By \ref{gfcor} and \ref{gfcorrmk}(2), $\mu$-nonforking restricted to ${\bf K'}$ forms a good ($\geq\mu$)-frame perhaps except symmetry (it actually satisfies symmetry by \ref{symcor2}(2) but we do not need this result here). ${\bf K'}$ is also $\mu$-tame because ${\bf K}$ is $\mu$-tame under \ref{assum1} and we can extend a model in $K_\mu$ to a limit model which is in $K'$. By \ref{BGvarfact}(1), $\mu$-nonforking can be extended to a good $(\geq\mu)$-frame for types of arbitrary length. 

Since the extended frame enjoys symmetry and right transitivity, by \ref{BGvarfact}(2) it satisfies left transitivity. We check the $\mu$-witness property for singletons: let $M\leq_uN$ both in $K'$, $p\in\gs(N)$. Suppose for any $M'$ with $M\leq_u M'\leq_u N$, $\nr{M'}\leq\nr{M}+\mu=\nr{M}$, we have $p\restriction M'$ does not $\mu$-fork over $M$. We need to show that $p$ does not $\mu$-fork over $M$. Without loss of generality assume $\nr{N}>\nr{M}$. By existence of $\mu$-nonsplitting (\ref{eeuprop}), there is $N'\in K_\mu$, $N'\leq N$ such that $p$ does not $\mu$-split over $N'$. As $N$ is saturated (replace ``$\mu$'' by $\nr{N}$ in \ref{unicor}(2)), we can obtain $N''\in K'_{\nr{M}}$ such that $N'<_uN''<_uN$ and $M\leq_uN''$. By definition $p$ does not $\mu$-fork over $N''$. Since $p\restriction N''$ does not $\mu$-fork over $M$ by assumption, \ref{gft} guaratees that $p$ does not $\mu$-fork over $M$.
\end{proof}

For comparison purposes, we reproduce the original proof of \ref{BGfact2} that uses existence for longer types. Then we give an alternative proof that bypasses it, so that we can utilize the frame in \ref{BGfactprop}.
\begin{proof}[Original proof of \ref{BGfact2}]
The forward direction is by definition of $U$-rank. For the backward direction, we show that for any ordinal $\alpha$, $U(p)\geq\alpha$ iff $U(q)\geq\alpha$. It suffices to consider the successor case: if $U(q)\geq\alpha+1$, then it has a forking extension $q'\in\gs(M_2)$ of rank $\geq\alpha$, with $\nr{M_2}=\nr{M_1}$. By monotonicity of nonforking, $q'$ is also a forking extension of $p$. However, $\nr{M}$ might not be the same as $\nr{M_2}$ (this was pointed out by \cite{GM}). We claim that there must be some $p'\in\gs(M')$ such that 
\begin{itemize}
\item $\nr{M'}=\nr{M}$;
\item $p\leq p'\leq q'$; and
\item $p'$ is a forking extension of $p$.
\end{itemize} Otherwise, every such $p'$ satisfying the first two requirements must be a nonforking extension of $p$. By $\ls$-witness property, $q'$ is also a nonforking extension of $p$, contradiction. Since $U(q')\geq\alpha$, by inductive hypothesis $U(p')\geq\alpha$, and hence $U(p)\geq\alpha+1$. 

If $U(p)\geq\alpha+1$, by definition there is $p'\in\gs(M_2)$ such that $\nr{M_2}=\nr{M}$ and $p'$ is a forking extension of $p$ of rank $\geq\alpha$. We claim that we can choose $p'$ and $M_2$ so that there is $q'\in\gs(M_3)$ with
\begin{itemize}
\item $q'$ extends $p$ and $p'$;
\item $M_3$ extends $M_1$ and $M_2$;
\item $q'$ is a nonforking extension of $p'$.
\end{itemize}
Assume that such $p'$ and $M_2$ are chosen, we show that $q'$ is a forking extension of $q$: otherwise by transitivity, $q'$ is a nonforking extension of $p$, and by monotonicity $p'$ is also a nonforking extension of $p$, contradiction. Now $q'$ is a nonforking extension of $p'$, so by inductive hypothesis $U(q')=U(p')\geq\alpha$. On the other hand, $q'$ is a forking extension of $q$, so by definition $U(q)\geq U(q')+1\geq\alpha+1$ as desired.

It remains to guarantee such $p'$ and $M_2$ above exist. Let $d$ realizes $q$ and $d'$ realizes $p'$. Since both $p'$ and $q$ extends $p$, there is $f\in\oop{Aut}_{M_0}(\mn)$ such that $f(d')=d$. Since $\gtp(d/M_1)$ does not fork over $M_0$, by symmetry there is $\bar{M_0}$ containing $M_0$ and $d$ such that $\gtp(M_1/\bar{M_0})$ does not fork over $M_0$. Let $\bar{M_1}$ extends both $\bar{M_0}$ and $M_1$ (possible because we work in $\mn$). By \emph{existence} $\gtp(f[M_2]/\bar{M_0})$ does not fork over $\bar{M_0}$. By extension there is $M_2^*$ such that $\gtp(M_2^*/\bar{M_1})$ does not fork over $\bar{M_0}$ and $\gtp(M_2^*/\bar{M_0})=\gtp(f[M_2]/\bar{M_0})$. Hence there is $g\in\oop{Aut}_{\bar{M_0}}(\mn)$ with $g[f[M_2]]=M_2^*$. We now invoke \ref{BGrewrite} where we substitute $N_0,N_1,\bar{N_0},\bar{N_1},N_2,c$ by $M_0,M_1,\bar{M_0},\bar{M_1},M_2^*,d$ respectively. Then we obtain some $M_3$ extending $M_1$ and $M_2^*$ such that $\gtp(d/M_3)$ does not fork over $M_2^*$. $p'\defeq\gtp(d/M_2^*)$ satisfies the requirements.
\end{proof}
\begin{proof}[Alternative proof of \ref{BGfact2}]
In the original proof, the only place that uses existence for longer types is to guarantee $\gtp(f[M_2]/\bar{M_0})$ does not fork over $\bar{M_0}$. Pick any $M_4\leq\mn$ that extends both $f[M_2]$ and $M_1$. We will work in the minimal closure of the independence relation and use \ref{closlem}. From the original proof, we have obtained $\gtp(M_1/\bar{M_0})$ does not fork over $M_0$. By monotonicity $\gtp(M_1/\bar{M_0})$ does not fork over $\bar{M_0}$. By symmetry (for the minimal clsoure), $\gtp(\bar{M_0}/M_1)$ does not fork over $\bar{M_0}$. By extension (see \cite[Definition 3.5]{BGKV}), there is $M^*$ and $f\in\oop{Aut}_{\bar{M_0}M_1}(\mn)$ such that $\gtp(M^*/M_4)$ does not fork over $\bar{M_0}$ and $f[\bar{M_0}]=M^*$. Since $f$ fixes $\bar{M_0}$, $M^*=\bar{M_0}$. Therefore, $\gtp(\bar{M_0}/M_4)$ does not fork over $\bar{M_0}$. By monotonicity, $\gtp(\bar{M_0}/f[M_2])$ does not fork over $\bar{M_0}$. Symmetry gives the desired result.
\end{proof}
\begin{corollary}
Let ${\bf K}$ be $\mu$-superstable and ${\bf K'}$ be the AEC of the limit models in $K_{\geq\mu}$ ordered by $\leq_u$. Let $\fork$ be the extended frame from \ref{BGfactprop} and define the $U$-rank for $\fork$. For any $M<_uM_1\in K'$, $p\in\gs(M)$, any $q\in\gs(M_1)$ extending $p$ such that both $U(p),U(q)<\infty$, then 
$$U(p)=U(q)\Leftrightarrow q\text{ is a nonforking extension of }p$$
\end{corollary}
\begin{proof}
Combine \ref{BGfact2} and \ref{BGfactprop}. The alternative proof of \ref{BGfact2} (given before \ref{BGfactprop}) shows that existence is not necessary.
\end{proof}

\section{The main theorems and applications}\label{ufsec8}
We summarize our results in two main theorems. The first one concerns stable AECs while the second one concerns superstable ones. Some of the following items allow $\mu\geq\ls$ but we assume $\mu>\ls$ for a uniform statement. The proofs will come after the main theorems.
\begin{mainthm}\mylabel{maint1}{Main Theorem \thetheorem}
Let ${\bf K}$ be an AEC with a monster model, $\mu>\ls$, $\delta\leq\mu$ both be regular. Suppose ${\bf K}$ is $\mu$-tame, stable in $\mu$ and has continuity of $\mu$-nonsplitting. The following statements are equivalent under extra assumptions specified after the list:
\begin{enumerate}
\item ${\bf K}$ has $\delta$-local character of $\mu$-nonsplitting;
\item There is a good frame over the skeleton of $(\mu,\geq\delta)$-limit models ordered by $\leq_u$, except for symmetry and local character $\delta$ in place of $\al$. In this case the frame is canonical;
\item ${\bf K}$ has uniqueness of $(\mu,\geq\delta)$-limit models;
\item For any increasing chain of $\mu^+$-saturated models, if the length of the chain has cofinality $\geq\delta$, then the union is also $\mu^+$-saturated;
\item $K_{\mu^+}$ has a $\delta$-superlimit.
\end{enumerate}
(1) and (2) are equivalent. If ${\bf K}$ is $(<\mu)$-tame, then (3) implies (1).  There is $\lambda_1<h(\mu)$ such that if ${\bf K}$ is stable in $[\mu,\lambda_1)$, then (1) implies (3). Given any $\zeta\geq\mu^+$, stability in $[\mu,\lambda_1)$ can be replaced by stability in $[\mu,\zeta)$ plus no $\mu$-order property of length $\zeta$.

There is $\lambda_2<h(\mu^+)$ such that if ${\bf K}$ is stable in $[\mu^+,\lambda_2)$ and has continuity of $\mu^+$-nonsplitting, then (1) implies (4). Given any $\zeta\geq\mu^{++}$, stability in $[\mu^+,\lambda_2)$ can be replaced by stability in $[\mu^+,\zeta)$ plus no $\mu^+$-order property of length $\zeta$. Always (4) and (5) are equivalent and they imply (3). 
\end{mainthm}
The following diagram summarizes the implications in \ref{maint1}. Labels on the arrows indicate the extra assumptions needed, in addition to a monster model, $\mu$-tameness, stability in $\mu$ and continuity of $\mu$-nonsplitting. As in the theorem statement, whenever we require stability in the form $[\xi,\lambda)$, we can replace it by stability in $[\xi,\zeta)$ plus no $\xi$-order property of length $\zeta$.
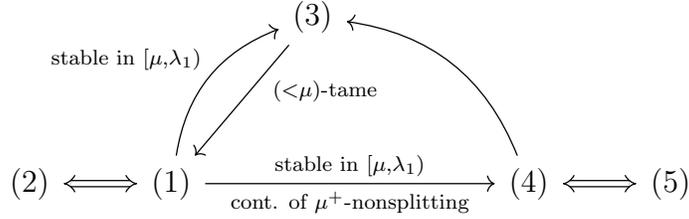
\begin{center}
\begin{tikzcd}
                          &                                                                                                                                & (3) \arrow[ldd, "(<\mu)\text{-tame}", near start] &    &                                   &     \\
                          &                                                                                                                               & &                                       &                                       &     \\
(2) \arrow[r, Leftrightarrow] & (1) \arrow[ruu, "{\text{stable in }[\mu,\lambda_1)}", bend left] \arrow[rrr,"{\text{stable in }[\mu,\lambda_1)}", "\text{cont. of }\mu^+\text{-nonsplitting}"'] &     &                                  & (4) \arrow[r, Leftrightarrow] \arrow[lluu, bend right] & (5)
\end{tikzcd}
\end{center}
\begin{mainthm}\mylabel{maint2}{Main Theorem \thetheorem}
Let ${\bf K}$ be an AEC with a monster model, $\mu>\ls$ be regular. Suppose ${\bf K}$ is $\mu$-tame, stable in $\mu$ and has continuity of $\mu$-nonsplitting. The following statements are equivalent modulo $(<\mu)$-tameness and a jump in cardinal (specified after the list):
\begin{enumerate}
\item ${\bf K}$ has $\al$-local character of $\mu$-nonsplitting;
\item There is a good frame over the limit models in $K_\mu$ ordered by $\leq_u$, except for symmetry. In this case the frame is canonical;
\item $K_\mu$ has uniqueness of limit models;
\item For any increasing chain of $\mu^+$-saturated models, the union of the chain is also $\mu^+$-saturated;
\item $K_{\mu^+}$ has a superlimit;
\item ${\bf K}$ is $(\mu^+,\mu^+)$-solvable;
\item ${\bf K}$ is stable in $\geq\mu$ and has continuity of $\mu^{+\omega}$-nonsplitting;
\item $U$-rank is bounded when $\mu$-nonforking is restricted to the limit models in $K_\mu$ ordered by $\leq_u$.
\end{enumerate}
(1), (2) and (8) are equivalent and each of them implies (3) and (4). If ${\bf K}$ is $(<\mu)$-tame, then (3) implies (1). Always (4) and (5) are equivalent and they imply (3). (1) implies (6) and (7) while (6) implies (4). (7) implies $(1)_{\mu^{+\omega}}$: ${\bf K}$ has $\al$-local character of $\mu^{+\omega}$-nonsplitting. %In the last direction, given any $\zeta>\mu^{+\omega}$, stability in $[\mu,\lambda)$ can be replaced by stability in $[\mu,\zeta)$ plus no $\mu^{+\omega}$-order property of length $\zeta$.

\end{mainthm}
The jump in cardinal is due to the lack of a precise bound on $\lambda'({\bf K})$ in deducing (7)$\Rightarrow$(1) (see \ref{lkques}(1)). The following diagram summarizes the implications in \ref{maint2}. ``$\mu^{+\omega}$'' indicates the jump in cardinal.
\begin{center}
\begin{tikzcd}
                           &                                                                                  & (3) \arrow[ldd, "(<\mu)\text{-tame}"] &                                                   &     \\
(2) \arrow[rd, Leftrightarrow] &                                                                                  &                                       &                                                   &     \\
(8)                        & (1) \arrow[ruu, bend left] \arrow[rr] \arrow[l, Leftrightarrow] \arrow[rd] \arrow[d] &                                       & (4) \arrow[r, Leftrightarrow] \arrow[luu, bend right] & (5) \\
                           & (7) \arrow[u, "\mu^{+\omega}", bend left]                                        & (6) \arrow[ru]                        &                                                   &    
\end{tikzcd}
\end{center}
\begin{proof}[Proof of \ref{maint1}]
(1) and (2) are equivalent by \ref{gfcor} and \ref{getloc}. The canonicity of the frame is by \ref{canon}. Suppose (3) holds. Then the proof of \ref{unicor}(2) and \ref{localfact}(1) give (1).

Suppose (1) holds. Obtain $\lambda_1=\lambda$ from \ref{unicor} and take $\chi=\delta$. If ${\bf K}$ is stable in $[\mu,\lambda_1)$, then it has uniqueness of $(\mu,\geq\delta)$-limit models, so (3) holds. The alternative hypotheses of stability and no-order-property work because we can replace $\lambda$ in the proof of \ref{sympf} by $\zeta$. 

The direction of (1) to (4) is by \ref{s19prop}. The alternative hypotheses work because we can replace $\lambda$ in the proof of \ref{sympf} by $\zeta$. (4) and (5) are equivalent by \ref{chainprop} and \ref{supsat}. They imply (3) by \ref{uufact}.
\end{proof}

For the proof of \ref{maint2}, we show the additional directions and refer the readers to the proof of \ref{maint1} for the original directions. 

\begin{proof}[Proof of \ref{maint2}]
Compared to \ref{maint1}, we do not need the extra stability and continuity of nonsplitting assumptions because superstability already implies them (\ref{chaincor}(1) and \ref{weaker}(1)). (1) and (8) are equivalent by \ref{urankcor}. (1) implies (7) by \ref{chaincor}(1) while (1) implies (6) by the forward direction of \ref{supsolv}. (6) plus $(<\mu)$-tameness implies (4) by the proof of the backward direction of \ref{supsolv}.  (7) implies $(1)_{\mu^{+\omega}}$ by \ref{cfcal}.
\end{proof}
\begin{remark}\mylabel{ufrmk}{Remark \thetheorem}
In \cite[Corollary 5.5]{GV}, they did not assume continuity of nonsplitting and showed that: if item (4) in \ref{maint2} holds in some $\xi\geq\beth_\omega(\chi_0+\mu)$ (see \ref{bvfact} for the definition of $\chi_0$), then every limit model in $K_{\xi}$ is $\beth_\omega(\chi_0+\mu)$-saturated. This implies $\al$-local character of $\xi$-nonsplitting. Using \cite[Theorem 7.1]{bv}, there is a $\lambda<h(\xi)$ such that (3) holds with $\mu$ replaced by $\lambda$. From hindsight, the last argument can be improved by quoting \ref{chaincor}(3) instead and having $\lambda=\xi^{+}$. In comparison, our (4)$\Rightarrow$(3) allows (3) to still be in $K_\mu$ and does not have the high cardinal threshold.
\end{remark}
\begin{corollary}
Let $\xi>\ls$ and ${\bf K}$ have a monster model, continuity of $\xi$-nonsplitting and be $(<\xi)$-tame. Then the following are equivalent:
\begin{enumerate}
\item ${\bf K}$ has uniqueness of limit models in $K_\xi$: for any $M_0,M_1,M_2\in K_\xi$, if both $M_1$ and $M_2$ are limit over $M_0$, then $M_1\cong_{M_0}M_2$;
\item ${\bf K}$ has uniqueness of limit models without base in $K_\xi$: any limit models in $K_\xi$ are isomorphic.
\end{enumerate}
\end{corollary}
\begin{proof}
The forward direction is immediate and only requires $JEP$. For the backward direction, the proof of (3)$\Rightarrow$(1) in \ref{maint2} goes through ($JEP$ is needed) and we have $\xi$-superstability. By (1)$\Rightarrow$(3) in \ref{maint2}, it has uniqueness of limit models in $K_\xi$.
\end{proof}
As applications, we present alternative proofs to the results in \cite{mmlimit} and \cite{ss} with stronger assumptions. In \cite{mmlimit}, limit models of abelian groups are studied. 
\begin{fact}
\begin{enumerate}
\item \mylabel{m1}{Fact \thetheorem}\cite[Definition 3.1, Fact 3.2]{mmlimit} Let ${\bf K}^{ab}$ be the class of abelian groups ordered by subgroup relation. Then ${\bf K}^{ab}$ is an AEC with $\oop{LS}({\bf K}^{ab})=\al$, has a monster model and is $(<\al)$-tame. 
\item \cite[Fact 3.3(2)]{mmlimit} ${\bf K}^{ab}$ is stable in all infinite cardinals.
\item \cite[Corollary 3.8]{mmlimit} ${\bf K}^{ab}$ has uniqueness of limit models in all infinite cardinals.
\end{enumerate}
\end{fact}
In the original proof of \ref{m1}(3), an explicit algebraic expression of limit models was obtained, so that limit models of the same cardinality are isomorphic to each other. In \cite[Remark 3.9]{mmlimit}, it was remarked that \cite{s19} could be used to obtain uniqueness of limit models for high enough cardinals (above $\geq\beth_{(2^{\al})^+}$). We write down the exact argument using known results. Then we present another proof that covers lower cardinals using results in this paper (but not any algebraic description of limit models).
\begin{proof}[First proof of \ref{m1}(3)]
In \ref{splsupfact}(1), pick $\xi\geq(\lambda'({\bf K}))^++\chi_1$ with $\cf(\xi)=\al$. By \ref{m1}(2), ${\bf K}^{ab}$ is stable in $\xi$. So the conclusion of \ref{splsupfact}(1) gives superstability in $\geq\lambda'({\bf K}^{ab})$. By \cite[Corollary 1.4]{vv} (which combines \cite[Fact 2.16, Corollary 6.9]{vv}), ${\bf K}^{ab}$ has uniqueness of limit models in $K^{ab}_{\geq\lambda'({\bf K}^{ab})}$. Notice that by \ref{hanffact}, $\lambda'({\bf K}^{ab})<h(\lambda({\bf K}^{ab}))=h(\al)=\beth_{(2^{\al})^+}$, so we can guarantee uniqueness of limit models above $\beth_{(2^{\al})^+}$.
\end{proof}
\begin{proof}[Second proof of \ref{m1}(3)]
By \ref{m1}(1)(2), ${\bf K}^{ab}$ is stable in $\al$ and is $(<\al)$-tame. The latter implies $\omega$-locality. By \ref{weaker}(2), ${\bf K}^{ab}$ has continuity of $\al$-nonsplitting. By \ref{cfcalrmk}(1), it is superstable in $\geq\al$. By \ref{unicor}(1) (or simply \cite[Corollary 1.4]{vv}), it has uniqueness of limit models in all infinite cardinals.
\end{proof}

We turn to look at a strictly stable AEC. 
\begin{fact}
\begin{enumerate}
\item \mylabel{m2}{Fact \thetheorem} \cite[Definition 4.1, Facts 4.2, 4.5]{mmlimit} Let ${\bf K}^{tf}$ be the class of torsion-free abelian groups ordered by pure subgroup relation. Then ${\bf K}^{tf}$ is an AEC with $\oop{LS}({\bf K}^{tf})=\al$, has a monster model and is $(<\al)$-tame. 
\item \cite[Fact 4.7]{mmlimit} ${\bf K}^{tf}$ is stable in $\lambda$ iff $\lambda^{\al}=\lambda$. In particular ${\bf K}^{tf}$ is strictly stable.
\item \cite[Corollary 4.18]{mmlimit} Let $\lambda\geq\aleph_1$. ${\bf K}^{tf}$ has uniqueness of $(\lambda,\geq\aleph_1)$-limit models.
\item \cite[Theorem 4.22]{mmlimit} Let $\lambda\geq\al$. Any $(\lambda,\al)$-limit model in ${\bf K}^{tf}$ is not algebraically compact. 
\item \cite[Lemmas 4.10, 4.14]{mmlimit} Let $\lambda\geq\aleph_1$. Any $(\lambda,\geq\aleph_1)$-limit model in ${\bf K}^{tf}$ is algebraically compact. Any two algebraically compact limit models in $K^{tf}_\lambda$ are isomorphic.
\end{enumerate}
\end{fact}
The original proof of the second part of \ref{m2}(3) uses an explicit algebraic expression of algebraically compact groups \cite[Fact 4.13]{mmlimit}. Using the results of this paper, we give a weaker version but without using any algebraic expression of algebraically compact groups.
\begin{proposition}
Assume CH. If for all stability cardinal $\lambda\geq\aleph_1$, ${\bf K}^{tf}$ does not have the $\lambda$-order property of length $\lambda^{+\omega}$, then for all such $\lambda$, it has uniqueness of $(\lambda,\geq\aleph_1)$-limit models.
\end{proposition}
\begin{proof}
By CH and \ref{m2}(2), ${\bf K}^{tf}$ is stable in $\aleph_1$. By \ref{m2}(1), ${\bf K}^{tf}$ is $(<\al)$-tame, hence it has $\omega$-locality. By \ref{weaker}(2), ${\bf K}^{tf}$ has continuity of $\aleph_1$-nonsplitting. \ref{394prop} and \ref{chitrans} give $\aleph_1$-local character of $\lambda$-nonsplitting for all stability cardinals $\lambda$. By \ref{supfact}(1), ${\bf K}^{tf}$ is stable in $[\lambda,\lambda^{+\omega})$. By \ref{unicor}(1) and \ref{comparermk}(1), ${\bf K}^{tf}$ has uniqueness of $(\lambda,\geq\aleph_1)$-limit models for all $\lambda\geq\aleph_1$.
\end{proof}
\begin{ques}
Is it true that ${\bf K}^{tf}$ does not have $\aleph_1$-order property of length $\aleph_\omega$? 
\end{ques}

For \ref{m2}(4), the original proof argued that uniqueness of limit models \emph{eventually} leads to superstability for large enough $\lambda$ (from an older result in \cite{GV}). Then a specific construction deals with small $\lambda$. In \cite[Remark 4.23]{mmlimit}, it was noted that \cite[Lemma 4.12]{s19} could deal with both cases of $\lambda$. We give a full proof here (the algebraic description of limit models is needed):
\begin{proof}[Proof of \ref{m2}(4)]
Let $\lambda\geq\al$ and $M$ be a $(\lambda,\al)$-limit model. Then ${\bf K}^{tf}$ is stable in $\lambda$ and by \ref{m2}(2) $\lambda>\al$. Suppose $M$ is algebraically compact, by \ref{m2}(5) and \ref{unicor}(2) $M$ is isomorphic to $(\lambda,\geq\aleph_1)$-limit models and is saturated. By \ref{localfact}(2) (where $\langle M_i:i\leq\al\rangle$ witnesses that $M$ is $(\lambda,\al)$-limit), $\al$-local character of $\lambda$-nonsplitting applies to $M$. Since $M$ is arbitrary, ${\bf K}^{tf}$ has $\al$-local character of $\lambda$-nonsplitting, which implies stability in $\geq\lambda$ by \ref{supfact}(2), contradicting \ref{m2}(2).
\end{proof}
\begin{remark}
\cite[Lemma 4.12]{s19} happened to work because we do not care about the case $\al$ (which is not stable) and we can always apply item (2) in \ref{localfact}.
\end{remark}

In \cite{ss}, $\aleph_0$-stable AECs with $\al$-$AP$, $\al$-$JEP$ and $\al$-$NMM$ were studied. They built a superlimit model in $\al$ by connecting limit models with sequentially homogeneous models \cite[Theorem 4.4]{ss}. Then they defined splitting over finite sets where types have countable domains and obtained finite character assuming categoricity in $\al$ \cite[Fact 5.3]{ss}. This allowed them to build a good $\al$-frame over models generated by the superlimit. These methods are absent in our paper because we studied AECs with a general $\ls$, and our splitting is defined for types over model-domains. 

In \cite[Corollary 5.9]{ss}, they showed the existence of a superlimit in $\aleph_1$ assuming weak $(<\al,\al)$-locality among other assumptions. We will strengthen the locality assumption to $\omega$-locality, and work in a monster model to give an alternative proof. This allows us to bypass the machineries in \cite{ss} that are sensitive to the cardinal $\al$, and the technical manipulation of symmetry in \cite[Section 3]{ss}. Also, our result extends to a general $\ls$.
\begin{proposition}
Let ${\bf K}$ is an $\al$-stable AEC with a monster model and has $\omega$-locality. Then there is a superlimit in $\aleph_1$. In general, let $\lambda\geq\ls$, and if ${\bf K}$ is stable in $\lambda$ instead of $\al$, then it has a superlimit in $\lambda^+$.
\end{proposition}
\begin{proof}
Apply \ref{maint2}(1)$\Rightarrow$(5) where $\mu=\ls$ (that direction does not require $\mu>\ls$). Notice that $\omega$-locality implies $\ls$-tameness.
\end{proof}
Tracing our proof, we require global assumptions of a monster model and $\omega$-locality in order to use our symmetry results, especially \ref{sympf}. We end this section with the following:
\begin{ques}
Instead of global assumptions like monster model and no-order-property, is it possible to obtain local symmetry properties in Section \ref{ufsec5} using more local assumptions?
\end{ques}
\bibliographystyle{alpha}
\bibliography{references}

{\small\setlength{\parindent}{0pt}
\textit{Email}: wangchil@andrew.cmu.edu

\textit{URL}: http://www.math.cmu.edu/$\sim$wangchil/

\textit{Address}: {Department of Mathematical Sciences, Carnegie Mellon University, Pittsburgh PA 15213, USA}
\end{document}